\documentclass[12pt]{article}
\usepackage[letterpaper,margin=0.85in]{geometry}

\usepackage[british]{babel}
\usepackage[latin9]{inputenc}
\usepackage{amssymb}
\usepackage{amsthm}
\usepackage{xcolor}
\usepackage{blindtext}
\usepackage{amsmath, amsfonts}
\usepackage{empheq}
\usepackage{mdframed}
\DeclareMathOperator{\spn}{span}

\DeclareMathOperator{\tr}{tr}

\DeclareMathOperator{\Id}{Id}

\DeclareMathOperator{\grad}{grad}

\usepackage{enumerate}
\numberwithin{equation}{section}
\usepackage{tabulary}
\usepackage{bbm}
\usepackage{verbatim}
\usepackage[title]{appendix}

\usepackage[nameinlink,noabbrev]{cleveref}

\makeatletter
\def\@seccntformat#1{\@ifundefined{#1@cntformat}%
   {\csname the#1\endcsname\quad}
   {\csname #1@cntformat\endcsname}
}
\makeatother

\usepackage{lmodern}
\usepackage{paralist}
  \usepackage{paralist}
  \usepackage{graphics} 
  \usepackage{epsfig} 
\usepackage{graphicx}
\usepackage{epic}
\usepackage[Option]{overpic}
\usepackage{epstopdf}
\usepackage[font={small,it}]{caption}
\usepackage{subcaption}
\graphicspath{{figures/}} 
\usepackage{float}
\usepackage{slashed}
\usepackage{mathrsfs}
\usepackage{oldgerm}
\usepackage{dsfont}
\usepackage{setspace}
\usepackage{mathtools}
\usepackage[all]{xy}

\newcommand{\rmd}{\mathrm{d}}
\newcommand{\rmD}{\mathrm{D}}
\renewcommand{\epsilon}{\varepsilon}

\theoremstyle{plain}
\newtheorem{theorem}{Theorem}[section]

\newtheorem{lemma}[theorem]{Lemma}
\newtheorem{proposition}[theorem]{Proposition}

\newcommand{\nwc}{\newcommand}
\nwc{\red}[1]{\textcolor{red}{#1}}

\theoremstyle{definition}
\newtheorem{definition}[theorem]{Definition}

\theoremstyle{definition}
\newtheorem{definition and lemma}[theorem]{Definition and Lemma}

\theoremstyle{remark}
\newtheorem{remark}[theorem]{Remark}

\newcommand{\tostar}{\overset{*}{\lower0.5em\hbox{$\smash{\scriptscriptstyle\rightharpoonup}$}}}

\nwc{\blue}[1]{\textcolor{blue}{#1}}

\def\Id{{\textnormal{Id}}}

\def\txta{{\textnormal{a}}}

\def\txtr{{\textnormal{r}}}

\newcommand{\be}{\begin{equation}}
\newcommand{\ee}{\end{equation}}
\newcommand{\benn}{\begin{equation*}}
\newcommand{\eenn}{\end{equation*}}
\newcommand{\bea}{\begin{eqnarray}}
\newcommand{\eea}{\end{eqnarray}}
\newcommand{\beann}{\begin{eqnarray*}}
\newcommand{\eeann}{\end{eqnarray*}}

\numberwithin{equation}{section}
 
\date{\today} 

\begin{document}
\author{Maximilian Engel\thanks{Department of Mathematics, Freie Universit{\"a}t Berlin, Arnimallee 6, D-14195 Berlin. Corresponding author, email address: maximilian.engel@fu-berlin.de}~, Christian Kuehn\thanks{Faculty of Mathematics, Technical University of Munich,
Boltzmannstr. 3, D-85748 Garching bei M\"{u}nchen.}~, Matteo Petrera\thanks{Institute for Mathematics, Technical University of Berlin, Stra{\ss}e des 17. Juni 136,
D-10623 Berlin.}~ and Yuri Suris\footnotemark[3]}
 
\title{Discretized Fast-Slow Systems \\ 
with Canards in Two Dimensions}

\maketitle

\begin{abstract}
We study the problem of preservation of maximal canards for time discretized fast-slow systems with canard fold points. In order to ensure such preservation, certain favorable structure preserving properties of the discretization scheme are required. Conventional schemes do not possess such properties. We perform a detailed analysis for an unconventional discretization scheme due to Kahan. The analysis uses the blow-up method to deal with the loss of normal hyperbolicity at the canard point. We show that the structure preserving properties of the Kahan discretization for quadratic vector fields imply a similar result as in continuous time, guaranteeing the occurrence of maximal canards between attracting and repelling slow manifolds upon variation of a bifurcation parameter. The proof is based on a 
Melnikov computation along an invariant separating curve, which organizes the dynamics of the map similarly to the ODE problem.
\end{abstract}


\section{Introduction}

In this paper, we study the effect of the time discretization upon systems of ordinary differential equations (ODEs) which exhibit the phenomenon called ``canards''. It takes place, under certain conditions, in singularly perturbed (slow-fast) systems exhibiting fold points. The simplest form of such a system is
\begin{align} 
\label{ODE_intro}
\begin{array}{r@{\;\,=\;\,}l}
x' & f(x,y,\lambda,\epsilon), \\
y' & \epsilon g(x ,y,\lambda,\epsilon),
\end{array}
\end{align}
where we interpret $\epsilon >0$ as a small time scale parameter, separating between the fast variable $x$ and the slow variable $y$. 
For $\lambda=0$, the origin is assumed to be a non-hyperbolic fold point, possessing an attracting slow manifold and a repelling slow manifold. One says that the system admits a \emph{maximal canard} if there are trajectories connecting the attracting and the repelling slow manifolds \cite{BenoitCallotDienerDiener,DuRo96,ks2001/3}. This is a non-generic phenomenon which only becomes generic upon including an additional parameter $\lambda$, for the region of $\lambda$'s which is exponentially narrow as $\epsilon\to 0$. This makes the study of maximal canards especially challenging.

Krupa and Szmolyan~\cite{ks2011} have analyzed maximal canards for equation~\eqref{ODE_intro} by using the \emph{blow-up method} which allows to effectively handle the non-hyperbolic singularity at the origin. The key idea to use the blow-up method~\cite{Du78,Du93} for fast-slow systems goes back to Dumortier and Roussarie~\cite{DuRo96}. They observed that non-hyperbolic singularities can be converted into partially hyperbolic one by means of an insertion of a suitable manifold, e.g.~a sphere, at such a singularity.  The dynamics on this inserted manifold are partially hyperbolic, and truly hyperbolic in its neighborhood. The dynamics on the manifold are usually analyzed in different charts. See e.g.~\cite[Chapter~7]{ku2015} for an introduction into this technique. A non-exhaustive list of different applications to planar fast-slow systems includes \cite{DeMaesschalckDumortier7,
DeMaesschalckDumortier4,DeMaesschalckWechselberger,GucwaSzmolyan,ks2011,KuehnUM,KuehnHyp}. 

The main ingredient for a proof of maximal canards in \cite{ks2011} is the existence of a constant of motion for the dynamics in the \emph{rescaling chart} in the blown-up space. This constant of motion can be used for a \emph{Melnikov method} to compute the separation of the attracting and repelling manifold under perturbations, in particular to find relations between parameters $\epsilon$ and $\lambda$ under which the manifolds intersect, leading to a maximal canard. The role of this constant of motion suggests that, in order to retain the existence of maximal canards, the right choice of the time discretization scheme becomes of a crucial importance. Indeed, one can show that conventional discretization schemes like the Euler method \emph{do not} preserve maximal canards. The concept of a structure preserving discretization method is necessary. We investigate time discretization of the ODE~\eqref{ODE_intro} via the Kahan method which has been shown to preserve various integrability attributes in many examples (and known also as Hirota-Kimura method in the context of integrable systems, see e.g.~\cite{Kahan93, PetreraSuris18}).
We apply the blow-up method, which so far has been mainly used for flows, to the discrete time fast-slow dynamical systems induced by the Kahan discretization procedure. We show that these dynamical systems exhibit maximal canards for $\lambda$ and $\epsilon$ related by a certain a functional relation existing in a region which exponentially narrow with $\epsilon\to 0$. Thus, we extend to the discrete time context the previously known feature of the continuous time systems, provided an intelligent choice of the discretization scheme. We would like to stress that, despite the similarity of results to the continuous time case, the techniques of the proofs for the discrete time had to be substantially modified. In particular, the arguments based on the conserved quantity cannot be directly transferred into the discrete time context, since the conserved quantities there are only formal (divergent asymptotic series). Thus, it turned out to be necessary to use more general arguments based on the existence of an invariant measure and an invariant separating curve characterized as a singular curve of an invariant measure. We use also a more general version of the Melnikov method, similar to the one presented in \cite{Wechselberger2002}. 

Note that the application of the blow-up method to the discrete-time problem of folded canards is a considerable extension compared to the Euler discretizations for transcritical singularities, as studied in \cite{EngelKuehn18}. The folded canard case has specific dynamic structure, as explained above, such that a structure-preserving discretization method is needed, now performing the blow-up for the rational Kahan mapping. Compared to \cite{EngelKuehn18}, the kind of map is different, the structure in the singular limit is richer, there is an additional parameter $\lambda$, also rescaled in the blow-up, and the type of result, namely a continuation of the critical object along a two-parameter curve, is new.

Based on observations of this paper, the employment of Kahan's method for a treatment of canards can also be found in \cite{EngelJardon}. 
There, the simplest canonical form for folded, pitchfork and transcritical canards is studied and the focus lies on the linearization along trajectories. While it is demonstrated that explicit Runge-Kutta methods cannot provide symmetry of entry-exit relations, the linearization along the Kahan scheme and similar symmetric, A-stable methods are shown to preserve the typical continuous-time behaviour. Hence, the discussion of symmetry and linear stability in \cite{EngelJardon} supplements the paper at hand; here, we establish the existence and extension of maximal canards along parameter combinations for the nonlinear problem of folded canards with additional quadratic perturbation terms, in particular using the blow-up technique.

The paper is organized as follows. Section~\ref{sec:conttime} recalls the setting of fast-slow systems in continuous time and summarizes the main result on maximal canards, Theorem~\ref{canard_classic}, with a short sketch of the proof, as given in \cite{ks2011}. In Section~\ref{sec:disctime}, we study the problem of a maximal canard for systems with folds in discrete time. We establish the Kahan discretization of the canard problem in Section~\ref{secKahanmethcan} and discuss the reduced subsystem of the slow time scale in Section~\ref{sec:slowflow}. In Section~\ref{sec:blowup}, we introduce the blow-up transformation for the discretized problem. We discuss the dynamics for the entering and exiting chart in Section~\ref{sec:KahanK1}, and for the rescaling chart in Section~\ref{sec:KahanK2}. In Section~\ref{sec:dyn_prop_K2}, we explore the dynamical properties of the Kahan map in the rescaling chart, including a formal conserved quantity, an invariant measure and an invariant separating curve. Following this, we conduct the Melnikov computation along the invariant curve in Section~\ref{sec:Melnikovnew}, leading to the proof of the main Theorem~\ref{canard_discrete}, which is the discrete-time analogue to Theorem~\ref{canard_classic}. Finally, we provide various numerical illustrations in Section~\ref{sec:numerics} and conclude with an outlook in Section~\ref{sec:conclusion}.

Thus, we succeeded in adding the problem of maximal canards to the recent list of results, where a geometric analysis shows that certain features of fast-slow systems with non-hyperbolic singularities can be preserved via a suitable discretization, including the cases of the fold singularity~\cite{ns2013}, the transcritical singularity \cite{EngelKuehn18} and the pitchfork singularity~\cite{ArciEngelKuehn19}. More broadly viewed, our results also provide a continuation of a line of research on discrete-time fast-slow dynamical systems, which includes the study of canard/delay behavior in iterated maps via normal form transformations~\cite{Neishtadt3}, non-standard analysis~\cite{Fruchard,Fruchard2}, renormalization~\cite{Baesens1}, Gevrey series~\cite{Baesens2}, complex-analytic methods~\cite{FruchardSchaefke1}, and phase plane partitioning~\cite{MiraShilnikov}.

\medskip

\textbf{Acknowledgments:} The authors gratefully acknowledge support by DFG (the Deutsche Forschungsgemeinschaft) via the SFB/TR 109
``Discretization in Geometry and Dynamics''. ME acknowledges support by Germany's Excellence Strategy -- The Berlin Mathematics Research Center MATH+ (EXC-2046/1, project ID: 390685689), and CK acknowledges support by a Lichtenberg Professorship of the VolkswagenFoundation.
\section{Maximal canard through a fold in continuous time} \label{sec:conttime}

\subsection{Fast-slow systems}

We start with a brief review and notation for continuous-time fast-slow systems. 
Consider a system of singularly perturbed ordinary differential equations (ODEs) 
of the form
\begin{align} \label{slowequ}
\renewcommand{\arraystretch}{2.0}
\begin{array}{rcrcl}
\epsilon \dfrac{\rmd x}{\rmd \tau} &=& \epsilon \dot{x} &=& f(x,y,\epsilon), \\
\dfrac{\rmd y}{\rmd \tau}&=&\dot{y} &=& g(x,y,\epsilon), \quad \ x \in \mathbb{R}^m, 
\quad y \in \mathbb R^n, \quad 0 < \epsilon \ll 1\,,
\end{array}
\end{align}
where $f,g,$ are $C^k$-functions with $k \geq 3$. Since $\epsilon$ is a small parameter, 
the variables $x$ and $y$ are often called the \textit{fast} and 
the \textit{slow} variables, respectively. The time variable $\tau$
in~\eqref{slowequ} is termed the \textit{slow} time scale. The change of 
variables to the \textit{fast} time scale $t:= \tau / \epsilon$ transforms the system~\eqref{slowequ} 
into ODEs
\begin{align} \label{fastequ}
\begin{array}{r@{\;\,=\;\,}r}
x' & f(x,y,\epsilon), \\
y' & \epsilon g(x,y,\epsilon).
\end{array}
\end{align}
To both systems \eqref{slowequ} and \eqref{fastequ} there correspond respective limiting problems for $\epsilon = 0$: 
the \textit{reduced problem} (or \textit{slow subsystem}) is given by
\begin{align} \label{redequ}
\begin{array}{r@{\;\,=\;\,}l}
0 & f(x,y,0), \\
\dot{y} & g(x,y,0), 
\end{array}
\end{align}
and the \textit{layer problem} (or \textit{fast subsystem}) is 
\begin{align} \label{layerequ}
\begin{array}{r@{\;\,=\;\,}l}
x' & f(x,y,0), \\
y' & 0.
\end{array}
\end{align}
The reduced problem~\eqref{redequ} can be understood as a dynamical system on 
the \textit{critical manifold} 
$$S_0= \{(x,y) \in \mathbb{R}^{m+n} \,:\, f(x,y,0) = 0 \}\,.$$
Observe that the manifold $S_0$ consists of equilibria of the layer 
problem~\eqref{layerequ}. $S_0$ is called \textit{normally hyperbolic} if for all $p\in S_0$
the matrix $\textnormal{D}_xf(p)\in\mathbb{R}^{m\times m}$  has no eigenvalues
on the imaginary axis. For a normally hyperbolic $S_0$, \textit{Fenichel 
theory} \cite{Fenichel4,Jones,ku2015,WigginsIM} implies that, for 
sufficiently small $\epsilon$, there is a locally invariant slow manifold $S_{\epsilon}$ 
such that the restriction of~\eqref{slowequ} to $S_{\epsilon}$ is a regular
perturbation of the reduced problem~\eqref{redequ}. Furthermore, it follows from 
Fenichel's perturbation results that $S_{\epsilon}$ possesses an invariant 
stable and unstable foliation, where the dynamics behave as a small perturbation 
of the layer problem~\eqref{layerequ}.

\subsection{Main result on maximal canards in slow-fast systems with a fold} \label{sec:contmain}

A challenging phenomenon is the breakdown of normal hyperbolicity of $S_0$ such that 
Fenichel theory cannot be applied. Typical examples of such a breakdown are found 
at bifurcation points $p\in S_0$, where the Jacobi matrix $\rmD_x f(p)$ has at least one 
eigenvalue with zero real part. The simplest examples are \emph{folds} in planar systems ($m=n=1$),
i.e., points $p=(x_0,y_0)\in\mathbb R^2$ (without loss of generality $p=(x_0,y_0)=(0,0)$) where  $\partial f/\partial x$ vanishes and in whose neighbourhood $S_0$ looks like a parabola.
The left part of $S_0$ (with $x<0$) is denoted by $S_a$ ($a$ for ``attractive''), while its right part (with $x>0$) is denoted by $S_r$ ($r$ for ``repelling''). These notations refer to the properties of dynamics of the layer problem in the region $y>0$ (see e.g.~\cite[Figure 8.1]{ku2015}). By standard Fenichel theory, for sufficiently small $\epsilon > 0$, outside of an arbitrarily small neighborhood of $p$, the manifolds $S_a$ and $S_r$ perturb smoothly to invariant manifolds $S_{a, \epsilon}$ and $S_{r, \epsilon}$.

In the following we focus on 
the particularly challenging problem of fold points admitting  \emph{maximal canards}. 
In this case, the critical curve $S_0=\{f(x,y,0)=0\}$ can be locally parametrized as $y=\varphi(x)$ such that the reduced dynamics on $S_0$ are given by
\be \label{reduceddyn}
\dot{x} = \frac{g(x, \varphi(x),0)}{\varphi'(x)}.
\ee
In our setting, the function at the right-hand side is smooth at the origin, so that the reduced flow goes through the origin via a maximal solution $x_0(t)$ of~\eqref{reduceddyn} with $x_0(0) =0$. 
The solution $(x_0(t),y_0(t))$ with $y_0(t)=\varphi(x_0(t))$ connects both parts $S_a$ and $S_r$ of $S_0$. However, there is no reason to expect that for $\epsilon>0$, the (extension of the) solution parametrizing $S_{a, \epsilon}$ will coincide with the (extension of the) solution parametrizing $S_{r, \epsilon}$, unless there are some special reasons, like symmetry, forcing such a coincidence.

\begin{definition}
We say that a planar slow-fast system admits a \emph{maximal canard}, if the extension of the attracting slow manifold $S_{a,\epsilon}$ coincides with the extension of a repelling slow manifold $S_{r,\epsilon}$. 
\end{definition}

{\bf Example.} \label{rem:Hepsilon_lambda0}
Consider the system
\begin{align} \label{K2_epsilon}
\begin{array}{r@{\;\,=\;\,}l}
\epsilon \dot x & - y +  x^2 , \\
\dot y & x,
\end{array}
\end{align}
corresponding to $f(x,y,\epsilon)=x^2-y$ and $g(x,y,\epsilon)=x$. For the reduced system ($\epsilon=0$) we obtain $y=\varphi(x)=x^2$ and $2x\dot{x} = x$, hence $\dot x = 1/2$ (regular at $x=0$). The solution $x_0(t)$ is given by $x_0(t)=\tau/2$ so that 
$$
(x_0(\tau),y_0(\tau))=\Big(\frac{\tau}{2},\frac{\tau^2}{4}\Big).
$$
Observe that the system is symmetric with respect to the reversion of time $\tau\mapsto -\tau$ simultaneously with $x\mapsto -x$. This ensures the existence of the maximal canard also for any $\epsilon>0$. In this particular example, one can easily find the maximal canard explicitly. Indeed, one can easily check that, for any $\epsilon>0$,
$$
(x_{0,\epsilon}(\tau),y_{0,\epsilon}(\tau))=\Big(\frac{\tau}{2},\frac{\tau^2}{4}-\frac{\epsilon}{2}\Big)
$$
is a solution of \eqref{K2_epsilon} which parametrizes the invariant set
\begin{equation} \label{invariant_epsi}
S_{\epsilon} = \left\{ (x,y) \in \mathbb{R}^2 \, : \, y = x^2 - \frac{\epsilon}{2} \right\},
\end{equation}
which consists precisely of the attracting branch $S_{a, \epsilon} = \left\{ (x,y) \in S_{\epsilon} \, : \, x < 0 \right\}$ and the repelling branch $S_{r, \epsilon} = \left\{ (x,y) \in S_{\epsilon} \, : \, x > 0 \right\}$, such that trajectories on $S_{\epsilon}$ go through $x=0$ with the speed $\dot x = \epsilon/2$. However, any generic perturbation of this example, e.g. with $g(x,y,\epsilon)=x+x^2$, will destroy its peculiarity and will not display a maximal canard.
\medskip

Thus, maximal canards are not a generic phenomenon in the above setting. 
 In order to find a context where they become generic, we have to consider families depending on an additional parameter $\lambda$:
\begin{align} \label{fastequlambda}
\begin{array}{r@{\;\,=\;\,}l}
x' & f(x,y,\lambda, \epsilon), \\
y' & \epsilon g(x,y,\lambda, \epsilon).
\end{array}
\end{align}
We assume that at $\lambda = \epsilon = 0$, the vector fields $f$ and $g$ satisfy the above conditions.
By a local change of coordinates, the problem can be brought into the canonical form
\begin{align} \label{normalform}
\begin{array}{r@{\;\,=\;\,}l}
x' & - y k_1(x,y, \lambda, \epsilon) +  x^2 k_2(x,y, \lambda, \epsilon)  + \epsilon k_3(x,y, \lambda, \epsilon), \\
y' & \epsilon(x k_4(x,y, \lambda, \epsilon) - \lambda k_5(x,y,\lambda, \epsilon) + y k_6(x,y, \lambda, \epsilon)),
\end{array}
\end{align}
where  
\begin{align} \label{higherorder}
\begin{array}{r@{\;\,=\;\,}l}
k_i(x,y, \lambda, \epsilon) & 1 + \mathcal{O}(x,y, \lambda, \epsilon)\,, \quad i=1,2,4,5,\\
k_i(x,y, \lambda, \epsilon) & \mathcal{O}(x,y, \lambda, \epsilon)\,, \quad i=3,6.
\end{array}
\end{align}
The main result on existence of maximal canards, as given in \cite[Theorem 3.1]{ks2011}, can be summarized as follows. 
Set
\begin{align} \label{expansionconstants}
\renewcommand{\arraystretch}{2.0}
\begin{array}{r@{\;\,=\;\,}l}
a_1 & \frac{\partial k_3}{\partial x}(0,0,0,0), \ a_2 = \frac{\partial k_1}{\partial x}(0,0,0,0), \ a_3 = \frac{\partial k_2}{\partial x}(0,0,0,0), \\
a_4 &\frac{\partial k_4}{\partial x}(0,0,0,0), \ a_5 =  k_6(0,0,0,0),
\end{array}
\end{align}
and 
\begin{equation} \label{constant}
C = \frac{1}{8} ( 4 a_1 - a_2 + 3 a_3 - 2 a_4 + 2 a_5).
\end{equation}

\begin{theorem} \label{canard_classic} 
Consider system~\eqref{normalform} such that the solution $(x_0(t),y_0(t))$ of the reduced problem for $\epsilon=0$, $\lambda=0$ connects $S_a$ and $S_r$. Assume that $C\neq 0$. Then there exist $\epsilon_0 > 0$ and a smooth function 
$$ 
\lambda_c(\sqrt{\epsilon})= - C \epsilon + \mathcal{O}(\epsilon^{3/2}),
$$ 
defined on $[0, \epsilon_0]$ such that for $\epsilon \in [0, \epsilon_0]$ 
there is a maximal canard, that is,  the extended  attracting slow manifold $S_{a, \epsilon}$ coincides with the extended repelling slow manifold $S_{r,\epsilon}$, if and only if $\lambda = \lambda_c(\sqrt{\epsilon})$. 
\end{theorem}
The main result of this paper will be a discretized version of Theorem~\ref{canard_classic} restricted to quadratic vector fields, proving for Kahan maps the extension of canards along a parameter curve, as opposed to \cite{EngelJardon} where only Example~\eqref{K2_epsilon} and its linearization are studied.

The proof of Theorem~\ref{canard_classic} is based on the \emph{blow-up technique}, transforming the singular problem to a manifold where the dynamics can be desingularized and studied in two different charts. 
The crucial step in the second chart $K_2$ is the continuation of center manifold connections via a Melnikov method based on an integral of motion $H$. 
In the Appendix~\ref{Appendix}, we summarize this procedure from \cite{ks2011}, adding several observations on the dynamics; its separatrix, its invariant measure and an alternative non-Hamiltonian expression that relates to the discrete-time proof we will provide in the following.


\section{Maximal canard for a system with a fold in discrete time} \label{sec:disctime}

\subsection{Kahan discretization of canard problem} \label{secKahanmethcan}
We discretize system~\eqref{normalform} with the \emph{Kahan method}. It was introduced in \cite{Kahan93} as an unconventional discretization scheme applicable to arbitrary ODEs with quadratic vector fields. 
It was demonstrated in~\cite{PetreraPfadlerSuris09, PetreraPfadlerSuris11, PetreraSuris18} and in~\cite{Celledonietal2013} that this scheme tends to preserve integrals of motion and invariant volume forms. 
There are few general results available to support this claim, in particular, two
general cases of preservation of invariant volume forms in~\cite[Section 2]{PetreraPfadlerSuris11}
and a similar result for Hamiltonian systems with a cubic Hamilton
function in~\cite{Celledonietal2013}. 
However, the number of particular results not covered
by any general theory and reviewed in the above references, is quite
impressive. Our study here will contribute an additional evidence, as
the result of Section~\ref{sec:invariant_measure} also belongs to this category, i.e., is not
covered by known general statements.

Consider an ODE with a quadratic vector field:
\begin{equation} \label{genODE}
z' = f(z) = Q(z) + B z + c,
\end{equation}
where each component of $Q: \mathbb R^n \to \mathbb R^n$ is a quadratic form, $B \in \mathbb R^{n \times n}$ and $c \in \mathbb R^n$. The Kahan discretization of this system reads as
\begin{equation} \label{genKahan}
\frac{\tilde z - z}{h} = \bar Q(z, \tilde z) + \frac{1}{2} B( z + \tilde z) + c,
\end{equation}
where
$$ \bar Q(z, \tilde z)  = \frac{1}{2} ( Q(z +\tilde z) - Q(z) - Q(\tilde z))$$ 
is the symmetric bilinear form such that $ \bar Q(z,z) = Q(z)$. Note that equation~\eqref{genKahan} is linear with respect to $\tilde z$ and therefore defines a \emph{rational} map $\tilde z = F_f (z, h)$, which approximates the time $h$ shift along the solutions of the ODE~\eqref{genODE}. Further note that $F_f^{-1}(z,h) = F_f (z, - h)$ and, hence, the map is \emph{birational}. An explicit form of the map $F_f$ defined by equation~\eqref{genKahan} is given by
\begin{equation} \label{Kahanexplicit}
\tilde z = F_f (z, h) = z + h\Big(\Id - \frac{h}{2} \rmD f(z)\Big)^{-1} f(z).
\end{equation}

In order to be able to apply the Kahan discretization scheme, we restrict ourselves to systems \eqref{slowequ}, \eqref{fastequ} which are quadratic, that is, to
\begin{align}  \label{ODE_for_Kahan_slow}
\renewcommand{\arraystretch}{1.3}
\begin{array}{r@{\;\,=\;\,}l}
\epsilon \dot{x} & - y +  x^2  + \epsilon a_1 x - a_2 x y, \\
\dot{y} & x  - \lambda + a_5 y  + a_4 x^2,
\end{array}
\end{align}
resp.
\begin{align}  \label{ODE_for_Kahan}
\renewcommand{\arraystretch}{1.3}
\begin{array}{r@{\;\,=\;\,}l}
x' & - y +  x^2  + \epsilon a_1 x - a_2 x y, \\
y' & \epsilon(x  - \lambda) + \epsilon a_5 y  + \epsilon a_4 x^2,
\end{array}
\end{align}
which corresponds to normal forms \eqref{normalform} with $k_1=1+a_2x$, $k_2=1$, $k_3=a_1x$, $k_4=1+a_4x$, $k_5=1$, and $k_6=a_5$. 

\begin{remark}
It was demonstrated in \cite[Proposition 1]{Celledonietal2013} that Kahan map \eqref{Kahanexplicit} coincides with the map produced by the following implicit Runge-Kutta scheme, when the latter is applied to a quadratic vector field $f$:
\begin{equation} \label{RungeKutta}
\frac{\tilde z - z}{h} =  - \frac{1}{2} f(z) + 2f\left( \frac{z+\tilde z}{2} \right) - \frac{1}{2} f(\tilde z).
\end{equation}
This opens the way of extending our present results for more general (not necessarily quadratic) systems \eqref{normalform}. In the present paper we restrict ourselves to the case \eqref{ODE_for_Kahan}, since the algebraic structure keeps the calculations clear and explicit and demonstrates the central methodological aspects of our proofs. However, we additionally apply the scheme~\eqref{RungeKutta} to the folded canard problem with cubic nonlinearity in Section~\ref{sec:numerics}, illustrating its numerical capacity beyond the quadratic case. A proof of maximal canards for the non-quadratic case remains an open problem for future work.
\end{remark}

\subsection{Reduced subsystem of the slow flow} \label{sec:slowflow}

Kahan discretization of \eqref{ODE_for_Kahan_slow} reads:
\begin{equation}\label{Kahan for slow}
\renewcommand{\arraystretch}{1.9}
\begin{array}{r@{\;\,=\;\,}l}
\dfrac{\epsilon}{h}(\tilde x - x) & - \dfrac{1}{2}(y +\tilde y)+x\tilde  x +  \dfrac{\epsilon a_1}{2}(x+\tilde x) - \dfrac{a_2}{2}(\tilde x y+x\tilde y), \\ 
\dfrac{1}{h}(\tilde y - y) & \dfrac{1}{2}(x +\tilde x) -\lambda+ \dfrac{a_5}{2} (y+\tilde y) + a_4 x\tilde x.
\end{array}
\end{equation} 
\begin{proposition}
The reduced system \eqref{Kahan for slow} with $\epsilon=0$ defines an evolution on a curve $$S_{0,h}=\big\{(x,y)\in \mathbb R^2: y=\varphi_{0,h}(x)\big\}$$ which supports a one-parameter family of solutions $x_h(n;x_0)$ with $x_h(0;x_0)=x_0$. For small $\epsilon>0$, this curve is perturbed to normally hyperbolic invariant curves $S_{a,h,\epsilon}$ resp. $S_{r,h,\epsilon}$ of the slow flow \eqref{Kahan for slow part} for $x<0$, resp. for $x>0$. 
\end{proposition}

For the simplest case $a_1=a_2=a_4=a_5=0$ and $\lambda=0$, 
\begin{equation}\label{Kahan for slow part}
\renewcommand{\arraystretch}{1.9}
\begin{array}{r@{\;\,=\;\,}l}
\dfrac{\epsilon}{h}(\tilde x - x) & - \dfrac{1}{2}(y +\tilde y)+x\tilde  x, \\ 
\dfrac{1}{h}(\tilde y - y) & \dfrac{1}{2}(x +\tilde x).
\end{array}
\end{equation} 
everything can be done explicitly. Straightforward computations lead to the following results. 

The reduced system
\begin{equation}\label{Kahan reduced}
\renewcommand{\arraystretch}{1.9}
\begin{array}{r@{\;\,=\;\,}l}
0 & - \dfrac{1}{2}(\tilde y + y)+\tilde x x, \\ 
\dfrac{1}{h}(\tilde y - y) & \dfrac{1}{2}(\tilde x + x)
\end{array}
\end{equation} 
has an invariant critical curve
\begin{equation} \label{Critical curve}
S_{0,h}=\Big\{ (x,y)\in\mathbb R^2: y=x^2-\frac{h^2}{8}\Big\}.
\end{equation}
The evolution on this curve is given by $\tilde x=x+\frac{h}{2}$, so that $x_h(n;x_0)=x_0+\frac{nh}{2}$. 

For the full system \eqref{Kahan for slow part}, the symmetry $x\mapsto -x$, $h\to -h$ ensures the existence of an invariant curve 
\begin{equation} \label{Kahan invariant curve}
S_{\epsilon,h}=\Big\{ (x,y)\in\mathbb R^2: y=x^2-\frac{\epsilon}{2}-\frac{h^2}{8}\Big\},
\end{equation}
whose parts with $x<0$, resp $x>0$ are the invariant curves $S_{a,h,\epsilon}$ resp. $S_{r,h,\epsilon}$. This curve supports solutions
with $x(n)=x_0+\frac{nh}{2}$. Thus, system \eqref{Kahan for slow part} exhibits a maximal canard. Our goal is to establish the existence of a maximal canard for system \eqref{Kahan for slow}.

\subsection{Blow-up of the fast flow} \label{sec:blowup}

Kahan discretization of the fast flow \eqref{ODE_for_Kahan} is the system \eqref{Kahan for slow} with $h\mapsto h\epsilon$:
\begin{equation}\label{Kahan for normalform}
\renewcommand{\arraystretch}{1.9}
\begin{array}{r@{\;\,=\;\,}l}
\dfrac{1}{h}(\tilde x - x) & - \dfrac{1}{2}(\tilde y + y)+\tilde x x +  \dfrac{\epsilon a_1}{2}(\tilde x+x) - \dfrac{a_2}{2}(\tilde x y+x\tilde y), \\ 
\dfrac{1}{h}(\tilde y - y) & \dfrac{\epsilon}{2}(\tilde x + x) -\epsilon \lambda+ \dfrac{\epsilon a_5}{2} (\tilde y +y) + \epsilon a_4 \tilde xx.
\end{array}
\end{equation}

We introduce a quasi-homogeneous blow-up transformation for the discrete time system, interpreting the step size $h$ as a variable in the full system. Similarly to the continuous time situation, the transformation reads
$$ x =  r \bar x, \quad y =  r^2 \bar y, \quad \epsilon = 
 r^2 \bar \epsilon, \quad \lambda = r \bar \lambda, \quad h = \bar h/r\,, $$
where $(\bar x, \bar y, \bar \epsilon, \bar \lambda,  r, \bar h) \in B := S^2 \times [-\kappa, \kappa]  \times [0, \rho] \times 
[0, h_0] $ for some $h_0, \rho, \kappa > 0$. The change of variables in 
$h$ is chosen such that the map is desingularized in the relevant charts. 

This transformation is a map $\Phi: B \to \mathbb R^5$. 
If $F$ denotes the map obtained from the time-discretization, the map $\Phi$ induces a map $\overline{F}$ on $B$ by $\Phi \circ \overline{F} 
\circ \Phi^{-1} = F$. Analogously to the continuous time case, we are using the 
charts $K_i$, $i=1,2$, to describe the dynamics. The chart $K_1$ (setting $\bar y =1$) focuses on 
the entry and exit of trajectories, and is given by
\begin{equation} \label{K1dis}
x = r_1 x_1, \quad y = r_1^2, \quad \epsilon = r_1^2 \epsilon_1, \quad \lambda = r_1 \lambda_1, \quad h = h_1/r_1\,.
\end{equation}
In the scaling chart $K_2$ (setting $\bar \epsilon =1$) the dynamics arbitrarily close to the origin are analyzed. 
It is given via the mapping
\begin{equation} \label{K2dis}
x = r_2 x_2, \quad y = r_2^2 y_2, \quad \epsilon = r_2^2 , \quad \lambda = r_2 \lambda_2, \quad h = h_2/r_2\,.
\end{equation}
The change of coordinates from $K_1$ to $K_2$ is denoted by $\kappa_{12}$ and, for $\epsilon_1 > 0$, is given by
\begin{equation} \label{kappa12}
x_2 = \epsilon_1^{-1/2} x_1, \quad y_2 = \epsilon_1^{-1}, \quad r _2= r_1 \epsilon_1^{1/2}, \quad \lambda_2 = \epsilon_1^{-1/2} \lambda_1, \quad h_2 = h_1 \epsilon_1^{1/2}\,.
\end{equation}
Similarly, for $y > 0$, the map $\kappa_{21} = \kappa_{12}^{-1}$ is given by
\begin{equation} \label{kappa21}
x_1 = y_2^{-1/2} x_2, \quad r_1 = y_2^{1/2} r_2, \quad \epsilon_1 = y_2^{-1}, \quad \lambda_1 = y_2^{-1/2} \lambda_2, \quad h_1 = h_2 y_2^{1/2}\,.
\end{equation}

\subsection{Dynamics in the entering and exiting chart $K_1$} 
\label{sec:KahanK1}

Here we extend the dynamical equations \eqref{Kahan for normalform} by
\begin{equation}\label{trivial evolution ep la h}
\tilde \epsilon=\epsilon, \quad \tilde \lambda=\lambda, \quad \tilde h=h, 
\end{equation}
and then introduce the coordinate chart $K_1$ by \eqref{K1dis}:
\begin{equation} 
x = r_1 x_1, \quad y = r_1^2, \quad \epsilon = r_1^2 \epsilon_1, \quad \lambda = r_1 \lambda_1, \quad h = h_1/r_1,
\end{equation}
defined on the domain
\begin{equation} \label{D1}
D_1 = \left\{ (x_1, r_1, \epsilon_1, \lambda_1, h_1) \in \mathbb{R}^5 : 0 \leq r_1 \leq \rho, \;\; 0 \leq \epsilon_1 \leq \delta, \;\; 0 \leq h_1 \leq  \nu\right\}.
\end{equation}
where $\rho, \delta,\nu > 0$ are sufficiently small.

To transform the map~\eqref{Kahan for normalform} into the coordinates of $K_1$, we start with the particular case $a_1=a_2=a_4=a_5=0$, generated by difference equations
\begin{equation}\label{Kahan pure}
\frac{1}{h} (\tilde x - x)= \tilde x x - \frac{1}{2}(\tilde y + y), \quad \frac{1}{h}(\tilde y - y) = \frac{\epsilon}{2}(\tilde x + x) - \epsilon\lambda, 
\end{equation}
supplied, as usual, by \eqref{trivial evolution ep la h}. Written explicitly, this is the map
\begin{equation} \label{map_Kahan}
\tilde{x}= \frac{P(x,y,\epsilon,\lambda,h)}{R(x,\epsilon,h)}, \quad
\tilde{y}= \frac{Q(x,y,\epsilon,\lambda,h)}{R(x,\epsilon,h)}, \quad
\tilde \epsilon=\epsilon, \quad \tilde \lambda=\lambda, \quad \tilde h=h, 
\end{equation}
where
\begin{eqnarray}
P(x,y,\epsilon,\lambda,h) & = & x - hy - \tfrac{h^2}{4} \epsilon x + \tfrac{h^2}{2}\lambda \epsilon, \\
Q(x,y,\epsilon,\lambda,h) & = & y - hyx - \tfrac{h^2}{2} \epsilon x^2 - h\lambda \epsilon + h^2 x \lambda \epsilon + h \epsilon x - \tfrac{h^2}{4} \epsilon y,\\
R(x,\epsilon,h) & = & 1- hx + \tfrac{h^2}{4} \epsilon.
\end{eqnarray}
Upon substitution $K_1$, we have:
\begin{eqnarray}
P(x,y,\epsilon,\lambda,h) & = & r_1P_1(x_1,\epsilon_1,\lambda_1,h_1), \\
Q(x,y,\epsilon,\lambda,h) & = & r_1^2Q_1(x_1,\epsilon_1,\lambda_1,h_1), \\
R(x,\epsilon,h) & = & R_1(x_1,\epsilon_1,\lambda_1,h_1), 
\end{eqnarray}
where
\begin{eqnarray}
P_1(x_1,\epsilon_1,\lambda_1,h_1) & = & x_1 - h_1 - \tfrac{h_1^2}{4} \epsilon_1 x_1 + \tfrac{h_1^2}{2}\lambda_1 \epsilon_1, \\
Q_1(x_1,\epsilon_1,\lambda_1,h_1) & = & 1 - h_1x_1 - \tfrac{h_1^2}{2} \epsilon_1 x_1^2 - h_1\lambda_1 \epsilon_1 + h_1^2 x_1 \lambda_1 \epsilon_1 + h_1 \epsilon_1 x_1 - \tfrac{h_1^2}{4} \epsilon_1,\\
R_1(x_1,\epsilon_1,h_1) & = & 1- h_1x_1 + \tfrac{h_1^2}{4} \epsilon_1.
\end{eqnarray}
Setting 
\begin{equation}
Y_1(x_1,\epsilon_1,\lambda_1,h_1)=\frac{Q_1(x_1,\epsilon_1,\lambda_1,h_1)}{R_1(x_1,\epsilon_1,h_1)},
\end{equation}
\begin{equation}
X_1(x_1,\epsilon_1,\lambda_1,h_1)=\frac{P_1(x_1,\epsilon_1,\lambda_1,h_1)}{Q_1(x_1,\epsilon_1,\lambda_1,h_1)^{1/2}R_1(x_1,\epsilon_1,h_1)^{1/2}},
\end{equation}
we come to the following expression for the map \eqref{map_Kahan} in the chart $K_1$:
\begin{align} \label{K1dynamics_Kahan}
\renewcommand{\arraystretch}{1.6}
\begin{array}{r@{\;\,=\;\,}l}
\tilde x_1 & X_1(x_1,\epsilon_1,\lambda_1,h_1), \\
\tilde{r}_1 & r_1 (Y_1(x_1, \epsilon_1, \lambda_1, h_1))^{1/2},  \\
\tilde \epsilon_1 &  \epsilon_1 (Y_1(x_1, \epsilon_1, \lambda_1, h_1))^{-1}, \\
\tilde \lambda_1 &  \lambda_1 (Y_1(x_1, \epsilon_1, \lambda_1, h_1))^{-1/2}, \\
\tilde{h}_1 & h_1 (Y_1(x_1, \epsilon_1, \lambda_1, h_1))^{1/2}.
\end{array}
\end{align}
Now it is straightforward to extend these results to the general case of the map \eqref{Kahan for normalform} with arbitrary constants $a_i$. For this, we observe:
\begin{itemize}
\item[--] in the first equation, the terms $y$ and $x^2$ on the right-hand side scale as $r_1^2$ and $r_1^2x_1^2$, while the terms $\epsilon x$ and $xy$ scale as $r_1^3\epsilon_1x_1$ and $r_1^3x_1$, respectively;
\item[--] in the second equation, the terms $\epsilon x$ and $\epsilon\lambda$ on the right-hand side scale as $r_1^3\epsilon_1x_1$ and $r_1^3\epsilon_1\lambda_1$, while the terms $\epsilon y$ and $\epsilon x^2$ scale as $r_1^4\epsilon_1$ and $r_1^4\epsilon_1x_1^2$, respectively.
\end{itemize}
Therefore, we can treat all terms involving $a_1, a_2, a_4, a_5$ as $ \mathcal{O}( r_1)$.  The resulting map is given by formulas analogous to \eqref{K1dynamics_Kahan}, with $X_1(x_1,\epsilon_1,\lambda_1,h_1)$, $Y_1(x_1,\epsilon_1,\lambda_1,h_1)$ replaced by certain functions 
$$
X_1(x_1,\epsilon_1,\lambda_1,h_1)+\mathcal O(r_1)\quad {\rm  and} \quad Y_1(x_1,\epsilon_1,\lambda_1,h_1)+\mathcal O(\epsilon_1r_1).
$$

We now analyze the dynamics of this map.
\begin{itemize}
\item The subset $\{r_1 = 0, \;\epsilon_1 = 0, \;\lambda_1 = 0\} \cap D_1$ is invariant, and on this subset we have $Y_1(x_1,r_1,\epsilon_1, \lambda_1, h_1) =1$, so that
$$ \tilde x_1 = \frac{x_1 - h_1}{1 - h_1 x_1}, \quad \tilde h_1 = h_1.$$
Hence, it contains two curves of fixed points 
$$p_{a,1}(h_1) = (-1,0,0,0, h_1) \quad \text{and} \ p_{r,1}(h_1) = (1,0,0,0, h_1).$$
We have:
\begin{equation*}
 \left | \frac{\partial \tilde x_1}{\partial x_1} (p_{a,1}(h_1)) \right |= \left |  \frac{1 - h_1}{1 + h_1} \right | <1, \quad
\left |  \frac{\partial \tilde x_1}{\partial x_1} (p_{r,1}(h_1)) \right | = \left |  \frac{1 + h_1}{ 1 - h_1 } \right | >1
\end{equation*}
for $h_1 \leq \nu < 1$, hence the point $p_{a,1}(h_1)$ is attracting in the $x_1$-direction and the point $p_{r,1}(h_1)$ is repelling in the $x_1$-direction. In all other directions, the multipliers of these fixed points are equal to 1.

\item Similarly, we have on $\{\epsilon_1 = 0, \lambda_1 = 0\} \cap D_1$ for small $r_1 > 0$: 
$$ \tilde x_1 = \frac{x_1 - h_1}{1 - h_1 x_1} + \mathcal{O}(r_1), \quad \tilde h_1 = h_1, \quad \tilde r_1 = r_1.$$
By the implicit function theorem, we can conclude that on $\{\epsilon_1 = 0, \; \lambda_1 = 0\} \cap D_1 $, there exist two families of normally hyperbolic (for $h_1>0$) curves of fixed points denoted as $S_{a,1}(h_1)$ and $S_{r,1}(h_1)$, parametrized by $r_1\in[0,\rho]$ and ending for  $r_1 = 0$ at  $p_{a,1}(h_1)$ and $p_{r,1}(h_1)$, respectively.  For the map \eqref{map_Kahan}, corresponding to difference equation \eqref{Kahan pure} (that is, to \eqref{Kahan for normalform} with all $a_i=0$), the $\mathcal{O}(r_1)$-term vanishes, and the above families are simply given by 
\begin{eqnarray*}
S_{a,1}(h_1) & = &  \{(-1,r_1,0,0, h_1): 0 \leq r_1 \leq \rho\} \cap D_1, \\ 
S_{r,1}(h_1)  & = &  \{(1,r_1,0,0, h_1) : 0 \leq r_1 \leq \rho\} \cap D_1.
\end{eqnarray*}
\item On the invariant set $\{r_1 = 0, \lambda_1 = 0\} \cap D_1$,   the dynamics of $x_1$, $\epsilon_1$ and $h_1$ are given by
\begin{align}  \label{r1l10}
\renewcommand{\arraystretch}{1.6}
\begin{array}{r@{\;\,=\;\,}l}
\tilde x_1 & X_1(x_1,\epsilon_1,0,h_1), \\
 \tilde \epsilon_1 & \epsilon_1 (Y_1(x_1, \epsilon_1, 0, h_1))^{-1}, \\
 \tilde h_1 & h_1 (Y_1(x_1, \epsilon_1,0, h_1))^{1/2}.
\end{array}
\end{align}
We compute the Jacobi matrices of the map~\eqref{r1l10} at $p_{a,1}(h_1)$ and $p_{r,1}(h_1)$, restricting to the invariant set $\{r_1 = 0, \lambda_1 = 0\} \subset D_1$,
\begin{align*}
&A_{a}:=\frac{\partial (\tilde x_1, \tilde \epsilon_1, \tilde h_1)}{\partial (x_1, \epsilon_1, h_1)} (p_{a,1}(h_1)) = \begin{pmatrix}
\frac{1-h_1}{1 + h_1} & \frac{-h_1}{2(1+h_1)}  & 0\\
0 & 1 & 0 \\
0 & - \frac{h_1^2}{2} & 1
\end{pmatrix},
\\ 
&A_{r}:=\frac{\partial (\tilde x_1, \tilde \epsilon_1, \tilde h_1)}{\partial (x_1, \epsilon_1, h_1)} (p_{r,1}(h_1)) = \begin{pmatrix}
\frac{1+h_1}{1 - h_1} & \frac{-h_1}{2(1-h_1)} & 0\\
0 & 1 & 0 \\
0 & \frac{h_1^2}{2} & 1
\end{pmatrix} \,.
\end{align*} 
The matrix $A_{a}$ has a two-dimensional invariant space corresponding to the eigenvalue 1, spanned by the vectors $v_{a}^{(1)}=(0,0,1)^\top$ and $ v_{a}^{(2)} = (-1,4,0)^{\top}$, such that 
$$
(A_a-I)v_{a}^{(1)}=0, \quad (A_a-I)v_{a}^{(2)}=-2h_1^2v_{a}^{(1)}.
$$ 
Similarly, the matrix $A_{r}$ has a two-dimensional invariant space corresponding to the eigenvalue 1, spanned by the vectors $v_{r}^{(1)}=(0,0,1)^\top$ and $v_{r}^{(2)} = (1,4,0)^{\top}$, such that 
$$
(A_r-I)v_{r}^{(1)}=0, \quad (A_r-I)v_{a}^{(2)}=-2h_1^2v_{r}^{(1)}.
$$ 
It is instructive to compare this with the continuous-time case $h_1\to 0$ (see, e.g., \cite[Lemma 2.5]{ks2011}), where both vectors $v_{a}^{(1)}$ and $v_{a}^{(2)}$ are eigenvectors of the corresponding linearized system, with $v_{a}^{(1)}$ being tangent to $S_{a,1}$ and $v_{a}^{(2)}$ corresponding to the center direction in the invariant plane $r_1=0$ (and similarly for $v_{r}^{(1)}$ and $v_{r}^{(2)}$). 
\end{itemize}

We summarize these observations into the following statement.
\begin{proposition} \label{centermanifolds}
For system~\eqref{K1dynamics_Kahan}, there exist a center-stable manifold $\widehat M_{a,1}$ and a center-unstable manifold $\widehat M_{r,1}$,  with the following properties: 
\begin{enumerate}
\item For $i = a, r$, the manifold $\widehat M_{i,1}$ contains the curve of fixed points $S_{i,1}(h_1)$ on $\{\epsilon_1 = 0, \  \lambda_1 = 0\} \subset D_1$, parametrized by $r_1$, and the center manifold $N_{i,1}$ whose branch for $\epsilon_1, h_1 > 0$ is unique (see Figure~\ref{fig:blownup} (b)). In $D_1$, the manifold $\widehat M_{i,1}$ is given as a graph $x_1 = \hat g_i (r_1, \epsilon_1, \lambda_1,h_1)$. 

\item For $i = a, r$, there exist two-dimensional invariant manifolds $M_{i,1}$ which are given as graphs $x_1 = g_i (r_1, \epsilon_1)$.
\end{enumerate}
\end{proposition}

\begin{proof}
The first part follows by standard center manifold theory (see, e.g., \cite{HPS77}). There exist two-dimensional center manifolds $N_{a,1}$ and $N_{r,1}$, parametrized by $h_1, \epsilon_1$, which at $\epsilon_1 =0$ coincide with the sets of fixed points
\begin{equation} \label{setoffixedpoints}
P_{a,1} = \{p_{a,1}(h_1)\,:\, 0 \leq h_1 \leq \nu\} \quad \text{and} \quad  P_{r,1} = \{p_{r,1}(h_1)\,:\, 0 \leq h_1 \leq \nu\},
\end{equation}
respectively (see Figure~\ref{fig:blownup} (b)).
Note that, by \eqref{r1l10},  on $\{r_1 = 0, \ \lambda_1 = 0, \ h_1 > 0\} \cap D_1$ we have $\tilde \epsilon_1 > \epsilon_1$ and $\tilde h_1 < h_1$ for $ x_1 \leq 0$. Hence, for $\delta$ small enough, the branch of the manifold $N_{a,1}$ on $\{r_1 = 0, \epsilon_1 > 0, \lambda_1 = 0, h_1 > 0\} \cap D_1$  is unique. On the other hand, we observe that for $x_1\geq \frac{1}{K}$ with a constant $K> 1$, we have $\tilde \epsilon_1 < \epsilon_1$ and $\tilde h_1 > h_1$, if and only if $h_1 < \frac{2K}{1+K^2}$. Thus, for $x_1$ from a neighborhood of $1$, we see that $\nu <\frac{2K}{1+K^2} < 1$ guarantees that, for $\delta$ small enough depending on $K$, the branch of the manifold $N_{r,1}$ on $\{r_1 = 0, \  \epsilon_1 > 0,\ \lambda_1 = 0, \ h_1 > 0\} \cap D_1$ is unique.
\smallskip

The second part follows from the invariances $\tilde r_1 \tilde \lambda_1 = r_1 \lambda_1$ and $\tilde h_1/\tilde r_1 = h_1 / r_1$, compare \cite[Proposition 3.3 and Figure 2]{EngelKuehn18} for details.
\end{proof}

\subsection{Dynamics in the scaling chart $K_2$} 
\label{sec:KahanK2}

Next, we investigate the dynamics in the scaling chart $K_2$, in order to find a trajectory connecting $\widehat M_{a,1}$ with $\widehat M_{r,1}$, or $ M_{a,1}$ with $ M_{r,1}$ respectively.
Recall from \eqref{K2dis} that in chart $K_2$ we have
\begin{equation} \label{K2dis2}
x = r_2 x_2, \quad y = r_2^2 y_2, \quad \epsilon = r_2^2 , \quad \lambda = r_2 \lambda_2, \quad h = h_2/r_2\,.
\end{equation}
In this chart and upon the time rescaling $t=t_2/r_2$, equation~\eqref{ODE_for_Kahan} takes the form 
\begin{align} \label{ODE_Kahan_K2}
\renewcommand{\arraystretch}{1.3}
\begin{array}{r@{\;\,=\;\,}l}
x'_2 & - y_2 +  x_2^2 + r_2 (a_1 x_2 -a_2x_2y_2), \\
y'_2 & x_2 - \lambda_2 + r_2 (a_4x_2^2 + a_5y_2),
\end{array}
\end{align}
where the prime now denotes the derivative with respect to $t_2$, compare \eqref{K2_hard}.
Since in this chart $r_2=\sqrt{\epsilon}$ is not a dynamical variable (remains fixed in time), we will not write down explicitly differential, resp. difference evolution equations for $\lambda_2=\lambda/\sqrt{\epsilon}$ and for $h_2=h\sqrt{\epsilon}$. We will restore these variables as we come to the matching with the chart $K_1$. 
The Kahan discretization of equation~\eqref{ODE_Kahan_K2} with the time step $h_2$ can be written as
\begin{align}\label{Kahan_expan}
\renewcommand{\arraystretch}{1.3}
\begin{array}{r@{\;\,=\;\,}l}
\tilde x_2  & F_{1}(x_2, y_2, h_2)  + r_2 \hat G_1(x_2, y_2, h_2) + \lambda_2 \hat J_1(x_2, h_2), \\
\tilde y_2  & F_{2}(x_2, y_2, h_2) + r_2 \hat G_2(x_2, y_2, h_2) + \lambda_2 \hat J_2(x_2, h_2),\end{array}
\end{align} 
On the blow-up manifold $r_2=0$, we are dealing with the simple model system
\begin{equation}\label{Kahan_model}
\frac{1}{h_2} (\tilde x_2 - x_2)= x_2\tilde x_2 - \frac{1}{2}(y_2 + \tilde y_2), \quad \frac{1}{h_2}(\tilde y_2 - y_2) = \frac{1}{2}(x_2 +\tilde x_2) - \lambda_2.
\end{equation} 
This yields the birational map
\begin{align} \label{K2dynamics_Kahan}
\renewcommand{\arraystretch}{2.2}
\begin{array}{r@{\;\,=\;\,}l}
\tilde x_2 & \dfrac{x_2 - h_2y_2 - \frac{h_2^2}{4}  x_2 + \frac{h_2^2}{2}\lambda_2}{ 1- h_2x_2 + \frac{h_2^2}{4}} ,  \\
\tilde y_2 & \dfrac{y_2  + h_2 x_2- h_2 x_2 y_2- h_2 \lambda_2- \frac{h_2^2}{2}  x_2^2  + h_2^2\lambda_2 x_2   - \frac{h_2^2}{4} y_2}{ 1- h_2x_2 + \frac{h_2^2}{4}} .
\end{array}
\end{align}
This gives the following expressions for the map $F =(F_1, F_2)$ and $\hat J =(\hat J_1, \hat J_2)$ in \eqref{Kahan_expan}:
\begin{align} \label{map_Kahan_simple}
\renewcommand{\arraystretch}{2.2}
\begin{array}{r@{\;\,=\;\,}l}
\tilde x_2 & F_1(x_2,y_2,h_2)= \dfrac{x_2 - h_2y_2 - \frac{h_2^2}{4} x_2}{ 1- h_2x_2 + \frac{h_2^2}{4}}, \\
\tilde y_2 & F_2(x_2,y_2,h_2)=\dfrac{y_2 + h_2 x_2- h_2x_2 y_2- \frac{h_2^2}{2} x_2^2  - \frac{h_2^2}{4} y_2}{ 1- h_2x_2 + \frac{h_2^2}{4}},
\end{array}
\end{align}
and 
\begin{align} \label{J}
\renewcommand{\arraystretch}{2.2}
\begin{array}{r@{\;\,=\;\,}l}
\hat  J_1(x_2, h_2) &  \dfrac{\frac{h_2^2}{2}}{1 - h_2 x_2 + \frac{h_2^2}{4}}, \\
\hat J_2(x_2, h_2) & \dfrac{-h_2+h_2^2 x_2}{1 - h_2 x_2 + \frac{h_2^2}{4}}.
\end{array}
\end{align}
Explicit expressions for the functions $\hat G_1$ and $\hat G_2$ can be easily obtained, as well, but are  omitted here due to their length.

\subsection{Dynamical properties of the model map in the scaling chart} \label{sec:dyn_prop_K2}

For a better readability, we omit index ``2'' referring to the chart $K_2$ starting from here. In particular, we write $x$, $y$, $r$, $\lambda$, $h$ for $x_2$, $y_2$, $r_2$, $\lambda_2$, $h_2$ rather than for the original variables (before rescaling). 
Similarly to the continuous-time case, we start the analysis in $K_2$ with
the case $\lambda=0$, $r = 0$ for $h > 0$ fixed. This means that we study the dynamics of the map given by $F$ \eqref{map_Kahan_simple},
\begin{equation} \label{map_Kahan_simple 0}
\renewcommand{\arraystretch}{2.2}
F: \quad \left\{
\begin{array}{r@{\;\,=\;\,}l}
\tilde x &  \dfrac{x - hy - \frac{h^2}{4} x}{ 1- hx + \frac{h^2}{4}}, \\
\tilde y & \dfrac{y + h x- hxy - \frac{h^2}{2} x^2  - \frac{h^2}{4} y}{ 1- hx + \frac{h^2}{4}},
\end{array} \right.
\end{equation}
which comes as the solution of the difference equation 
\begin{equation}\label{Kahan_model 0}
\frac{1}{h} (\tilde x - x)= x\tilde x - \frac{1}{2}(y+ \tilde y), \quad \frac{1}{h}(\tilde y - y) = \frac{1}{2}(x +\tilde x).
\end{equation} 
We discuss in detail the most important properties of the model map \eqref{map_Kahan_simple 0}.
\subsubsection{Formal integral of motion} \label{sec:conservedquant}

Recall that, for $r = \lambda = 0$, the ODE system~\eqref{K2_easy} in the chart $K_2$ has a conserved quantity \eqref{firstintegral}. Its
level set $H(x,y)=0$ supports the special canard solution \eqref{specialsolution},
\begin{equation*}
\gamma_{0,2}(t_2) = \Big(\frac{1}{2} t_2, \frac{1}{4}  t_2^2 - \frac{1}{2}\Big)^\top.
\end{equation*}

In general, Kahan discretization has a distinguished property of possessing a conserved quantity for unusually numerous instances of quadratic vector fields. For \eqref{K2_easy}, it turns out to possess a formal conserved quantity in the form of an asymptotic power series in $h$. However, there are indications that this power series is divergent, so that map $F$ \eqref{map_Kahan_simple 0} does not possess a true integral of motion. Nevertheless, it possesses all nice properties of symplectic or Poisson integrators, in particular, a truncated formal integral is very well preserved on very long intervals of time. Moreover, as we will now demonstrate, the zero level set of the formal conserved quantity supports the special family of solutions of the discrete time system crucial for our main results.
\smallskip

We recall a method for constructing a formal conserved quantity
\begin{equation} \label{Hhexpansion}
\bar H( z, h) = H(z) + h^2 H_2(z) + h^4 H_4(z) + h^6 H_6(z) + \dots 
\end{equation}
for the the Kahan discretization $F_f$~\eqref{Kahanexplicit} for an ODE of the form~\eqref{genODE} admitting a smooth conserved quantity $H: \mathbb{R}^n \to \mathbb{R}$. The latter means that
\begin{equation} \label{conservation}
\sum_{i=1}^n \frac{\partial H(z)}{\partial z_i} f_i(z) = 0.
\end{equation}
The ansatz \eqref{Hhexpansion} containing only even powers of $h$ is justified by the fact that the Kahan method is a symmetric linear discretization scheme. Writing $\tilde z =F_f(z,h)$, we formulate our requirement of $\bar H$ being an integral of motion for $F_f$ as $\bar H(z,h)=\bar H( \tilde z, h)$ on $ \mathbb{R}^n \times [0,h_0]$, i.e., up to terms $ \mathcal{O}(h^4)$,
\begin{equation} \label{barHexpansion}
H(\tilde z) + h^2 H_2(\tilde z)  = H(z) + h^2 H_2(z) + \mathcal{O}(h^4).
\end{equation}
To compute the Taylor expansion of the left hand side, we observe:
\begin{align*}
H(\tilde z) &= H\left(z + h f(z) + \frac{h^2}{2} f(z)\rmD f(z) + \mathcal{O}(h^3)\right) \\
&= H(z) + h \sum_{i=1}^n \frac{\partial H(z)}{\partial z_i} f_i(z) \\
&+ \frac{h^2}{2} \left( \sum_{i,j=1}^n \frac{\partial^2 H(z)}{\partial z_i \partial z_j} f_i(z) f_j(z) + \sum_{i,j=1}^n \frac{\partial H(z)}{\partial z_i} \frac{\partial f_i(z)}{\partial z_j}  f_j(z) \right) + \mathcal{O}(h^3).
\end{align*}
Here, the $h$ and the $h^2$ terms vanish, as follows from \eqref{conservation} and its Lie derivative:
\begin{equation} \label{conservation_lie}
0 = \sum_{j=1}^n \frac{\partial}{\partial z_j}\left(\sum_{i=1}^n \frac{\partial H(z)}{\partial z_i} f_i(z)\right) f_j(z) = \sum_{i,j=1}^n \frac{\partial^2 H(z)}{\partial z_i \partial z_j} f_i(z) f_j(z) + \sum_{i,j=1}^n \frac{\partial H(z)}{\partial z_i} \frac{\partial f_i(z)}{\partial z_j}  f_j(z).
\end{equation}
Thus, we find: $H(\tilde z) = H(z) + \mathcal{O}(h^3)$, or, more precisely,
\begin{equation} \label{Hexpansion}
H(\tilde z) = H(z) + h^3 G_3(z) + h^4 G_4(z) + h^5 G_5(z) + \dots.
\end{equation}
Plugging this, as well as a Taylor expansion of $H_2(\tilde z)$ similar to $H(\tilde z)$, into \eqref{barHexpansion}, we see that vanishing of the $h^3$ terms is equivalent to
\begin{equation} \label{DiffEqu}
\sum_{i=1}^n\frac{\partial H_2(z)}{\partial z_i} f_i(z)  = -G_3(z).
\end{equation}
This is a linear PDE defining $H_2$ up to an additive term which is an arbitrary function of $H$.

Following terms $H_4, H_6, \dots$ can be determined in a similar manner, from linear PDEs like \eqref{DiffEqu} with recursively determined functions on the right hand side.

We now apply this scheme to obtain (the first terms of) the formal conserved quantity $\bar H(x,y,h)$ for~\eqref{map_Kahan_simple}. It turns out to be possible to find it in the form
\begin{equation} \label{barH}
\bar H(x,y, h) \approx H(x,y) +  \sum_{k=1}^\infty h^{2k} H_{2k}(x,y),
\end{equation}
where 
\begin{equation}
H(x, y) = e^{-2 y} \big( y - x^2 + \frac{1}{2} \big)\quad {\rm and} \quad H_{2k}(x,y)=e^{-2y}\bar H_{2k}(x,y),
\end{equation}
with $\bar H_{2k}(x,y)$ being polynomials of degree $2k+2$. The symbol $\approx$ reminds that this is only a formal asymptotic series which does not converge to a smooth conserved quantity. A Taylor expansion of $H(\tilde x, \tilde y)$ as in~\eqref{Hexpansion} gives
$$ 
H(\tilde x, \tilde y) = H( x,  y) + h^3 G_3(x,y) + \mathcal{O}(h^4),
$$
with
$$ 
G_3(x,y) = \frac{1}{3} e^{-2 y}  (x^3 + x^5 - 4 x^3 y + 3 xy^2).
$$
The differential equation~\eqref{DiffEqu} reads in the present case:
\begin{equation} \label{DiffEquspec}
 (x^2 - y)\frac{\partial }{\partial x}\big(e^{-2 y} \bar H_2(x,y)\big) + x\frac{\partial }{\partial y} \big(e^{-2 y} \bar H_2(x,y)\big)=-G_3(x,y). 
\end{equation} 
A solution for $\bar H_2$ which is a polynomial of degree 4 reads:
\begin{equation} \label{H2}
\bar H_2(x,y) = \frac{1}{3} \big(x^2-\frac{x^4}{2} + (y- x^2)(y -y^2)\big).
\end{equation}  
Hence, we obtain the approximation
\begin{equation} \label{barH2}
\bar H(x,y, h) =e^{-2 y} \big( y - x^2 + \frac{1}{2} \big) + \frac{h^2}{3} e^{-2y}  \big(x^2-\frac{x^4}{2} + (y- x^2)(y -y^2)\big) + \mathcal{O}(h^4).
\end{equation}
A straightforward computation shows that on the curve $y-x^2+\frac{1}{2}=0$ (the level set $H(x,y)=0$), the function $\bar H_2(x,y)$ takes a constant value $\frac{1}{8}$. Therefore, the level set $\bar H(x,y,h)=0$ is given, up to $\mathcal{O}(h^4)$, by
\begin{equation} \label{gamma from H2}
\varphi_{h}(x,y)=y - x^2 + \frac{1}{2}+\frac{h^2}{8} =0.
\end{equation}
Remarkably, we have the following statement.
\begin{proposition}
The curve \eqref{gamma from H2} represents a zero level set of the (divergent) formal integral $\bar H(x,y,h)$. More precisely, on this curve
$$
H(x,y)+\sum_{k=1}^n h^{2k} H_{2k}(x,y)=\mathcal O(h^{2n+2}).
$$
\end{proposition}
We will not prove this statement, but rather derive a different dynamical characterization of the curve \eqref{gamma from H2}. 

\subsubsection{Invariant measure} \label{sec:invariant_measure}

\begin{proposition}
The map $F$ given by \eqref{map_Kahan_simple 0} admits an invariant measure
\begin{equation} \label{Kahan_invariant_measure}
\mu_{h} = \frac{\rmd x \wedge \rmd y}{|\varphi_{h}(x,y)|} 
\end{equation}
with $\varphi_{h}(x,y)$ given in \eqref{gamma from H2}. This measure $\mu_h$ is singular on the curve $\varphi_h(x,y)=0$.
\end{proposition}
\begin{proof}
Difference equations \eqref{Kahan_model 0} can be written as a linear system for $(\tilde x,\tilde y)$:
$$
\renewcommand{\arraystretch}{1.3}
\begin{pmatrix}
1 - h x & \frac{h}{2} \\ 
-\frac{h}{2} & 1
\end{pmatrix} 
\begin{pmatrix}
\tilde x \\ \tilde y
\end{pmatrix} = 
\begin{pmatrix}
x - \frac{h}{2}y \\ y + \frac{h}{2}x
\end{pmatrix}.$$
Differentiating with respect to $x,y$, we obtain:
$$
\renewcommand{\arraystretch}{1.3}
\begin{pmatrix}
1 - h x & \frac{h}{2} \\
-\frac{h}{2} & 1
\end{pmatrix} \begin{pmatrix}
 \frac{\partial \tilde x}{\partial x} & \frac{\partial \tilde x}{\partial y} \\
\frac{\partial \tilde y}{\partial x} & \frac{\partial \tilde y}{\partial y} 
\end{pmatrix} = \begin{pmatrix}
1 + h \tilde x & - \frac{h}{2} \\
\frac{h}{2} & 1
\end{pmatrix}.
$$
Computing determinants, we find:
\begin{equation}\label{det J prelim}
 \det \frac{\partial(\tilde x, \tilde y)}{\partial (x,y)} = \frac{1 + h \tilde x + \frac{h^2}{4}}{1 - h x + \frac{h^2}{4}}.
\end{equation}

Next, we derive from the first equation in \eqref{map_Kahan_simple 0}:
$$
\tilde x-x=\frac{-hy+hx^2-\frac{h^2}{2}x}{1-hx+\frac{h^2}{4}}.
$$
Since the system \eqref{Kahan_model 0} is symmetric with respect to interchanging $(x,y)\leftrightarrow(\tilde x, \tilde y)$ with the simultaneous change $h\mapsto -h$, we can perform this operation in the latter equation, resulting in
$$
x-\tilde x=\frac{h\tilde y-h\tilde x^2-\frac{h^2}{2}\tilde x}{1+h\tilde x+\frac{h^2}{4}}.
$$
Comparing the last two formulas, we obtain:
$$
\frac{y-x^2+\frac{h}{2}x}{1-hx+\frac{h^2}{4}}=\frac{\tilde y-\tilde x^2-\frac{h}{2}\tilde x}{1+h\tilde x+\frac{h^2}{4}},
$$
or, equivalently,
\begin{equation}
\frac{y-x^2+\frac{1}{2}+\frac{h^2}{4}}{1-hx+\frac{h^2}{4}}=\frac{\tilde y-\tilde x^2+\frac{1}{2}+\frac{h^2}{4}}{1+h\tilde x+\frac{h^2}{4}}.
\end{equation}
Together with \eqref{det J prelim}, this results in
\begin{equation}\label{det J} 
 \det \frac{\partial(\tilde x, \tilde y)}{\partial (x,y)}  = \frac{\varphi_{h}(\tilde x, \tilde y)}{\varphi_{h}(x,y)},
\end{equation}
which is equivalent to the statement of proposition. 
\end{proof}

\subsubsection{Invariant separating curve}

It turns out that the singular curve of the invariant measure $\mu_h$ is an invariant curve under the map \eqref{map_Kahan_simple 0}.
\begin{proposition} \label{prop:Kahan_invariant}
The parabola
\begin{equation} \label{Kahan_invariant}
S_{h} : = \left\{ (x,y) \in \mathbb{R}^2 \, : \, y = x^2 - \frac{1}{2} - \frac{h^2}{8} \right\} 
\end{equation}
is invariant under the map $F$ given by \eqref{map_Kahan_simple 0}. Solutions on $S_{h}$ are given by
\begin{equation}\label{Kahan_sol_x0}
\renewcommand{\arraystretch}{2.2}
\gamma_{h, x_0}(n) = \begin{pmatrix} x_0+ \dfrac{hn}{2} \\ x_0^2+hnx_0+  \dfrac{h^2n^2}{4} - \dfrac{1}{2} - \dfrac{h^2}{8} \end{pmatrix}, \quad n \in \mathbb{Z}.
\end{equation}
For $(x,y)\in S_h$, we have: 
\begin{equation}\label{ineq x}
\left|\frac{\partial\tilde x}{\partial x}\right|  \quad \left\{ \begin{array}{ll} < 1 & \mathrm{for\;\;} x < 0, \\ =1 & \mathrm{ for \;\;}x=0,\\ >1 & \mathrm{for\;\;} x>0. \end{array}\right.
\end{equation}
\end{proposition}
\begin{proof}
Plugging $y= x^2 - \frac{1}{2} - \frac{h^2}{8}$ into formulas  \eqref{map_Kahan_simple 0}, we obtain upon a straightforward computation:
$$
\tilde x=x+\frac{h}{2}, \quad \tilde y= \Big(x+\frac{h}{2}\Big)^2-\frac{1}{2}-\frac{h^2}{8}.
$$
This proves the first two claims.

As for the last claim, we compute by differentiating the first equation in \eqref{map_Kahan_simple 0}:
\begin{equation} \label{KahanJacobian}
\frac{\partial \tilde x}{\partial x}   = 
\frac{1-h^2y - \frac{h^4}{16}}{\left(1-hx+\frac{h^2}{4} \right)^2}.
\end{equation}
For $(x,y)\in S_h$, this gives:
$$ 
\frac{\partial \tilde x}{\partial x} =  \dfrac{\left(1+\frac{h^2}{4}  \right)^2- h^2 x^2 }{ \left(1-hx+\frac{h^2}{4} \right)^2} 
= \dfrac{1+hx + \frac{h^2}{4}}{1- hx +\frac{h^2}{4}},
$$
which implies inequalities \eqref{ineq x}. (We remark that the right hand side tends to infinity as $x\to (1+\frac{h^2}{4})/h$.)
\end{proof}

The invariant set $S_h$ \eqref{Kahan_invariant} plays the role of a separatrix for $F$ \eqref{map_Kahan_simple}: bounded orbits of $F$ lie above $S_h$, while unbounded orbits of $F$ lie below $S_h$, as illustrated in Figures~\ref{fig:kahan_1}, \ref{fig:kahan_2}.

\begin{figure}[H]
\centering
\begin{overpic}[width=0.7\linewidth]{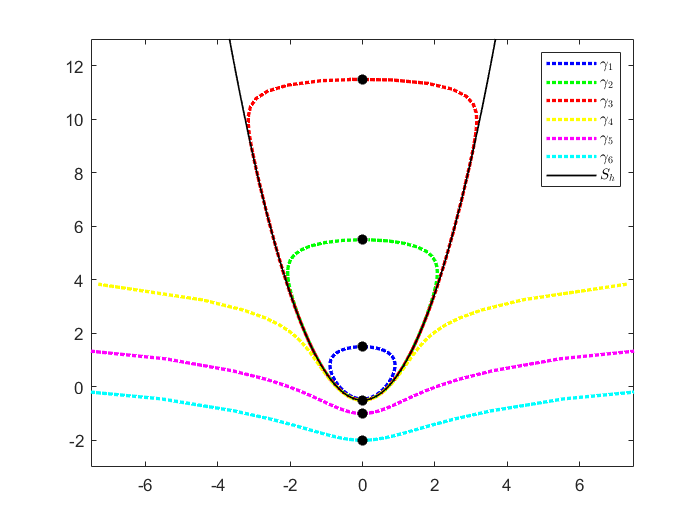}
 \put(54,2){$x$}
      \put(6,40){ $y$}
      \end{overpic}
  \caption{Trajectories for the Kahan map $F$ in chart $K_2$ \eqref{map_Kahan_simple} with $h = 0.01$ for different initial points $(x_{2,0},y_{2,0})$ (black dots): three bounded orbits above the separatrix $S_{h}$, and three unbounded orbits below the separatrix $S_h$. }
\label{fig:kahan_1}
\end{figure}
\begin{figure}[H]
\centering
\captionsetup[sub]{justification=centering}
\begin{subfigure}{0.45\textwidth}
\centering
\begin{overpic}[width=0.8\linewidth]{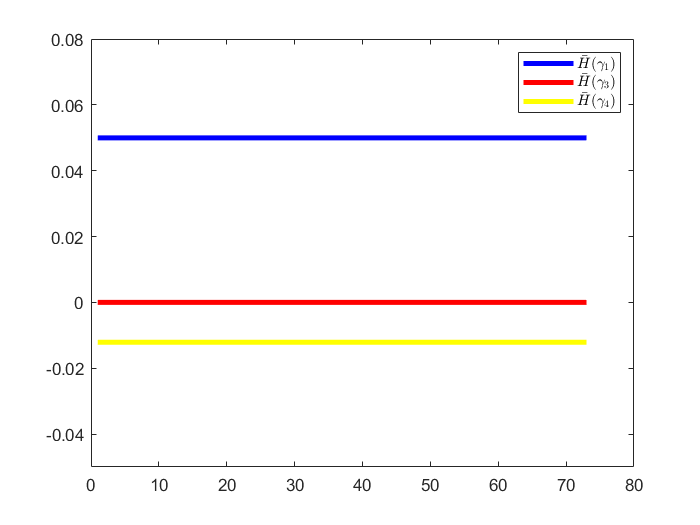}
 \put(45,0){\scriptsize time $n$}
\end{overpic}
 \caption{ }
 \end{subfigure}
 \begin{subfigure}{.45\textwidth}
 \centering
\begin{overpic}[width=0.8\linewidth]{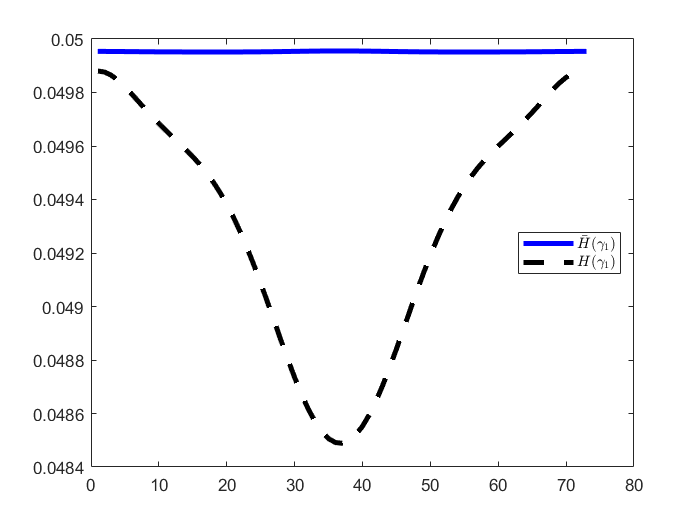}
 \put(45,0){\scriptsize time $n$}
\end{overpic}
  \caption{}
 \end{subfigure}
  \begin{subfigure}{.45\textwidth}
  \centering
\begin{overpic}[width=0.8\linewidth]{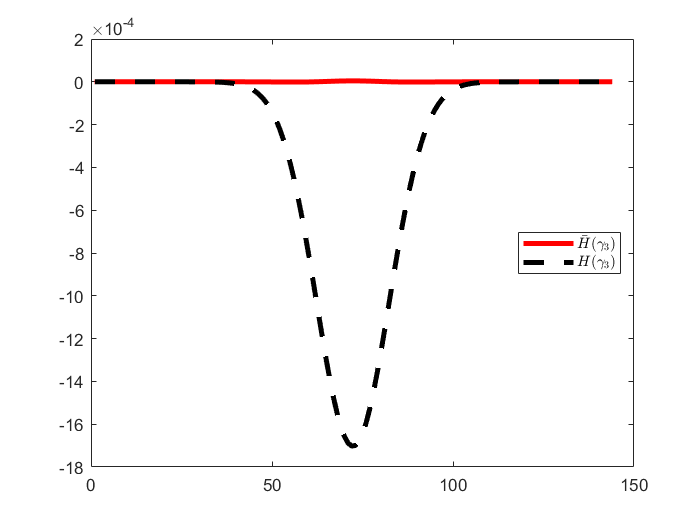}
 \put(45,0){\scriptsize time $n$}
\end{overpic}
  \caption{}
 \end{subfigure}
   \begin{subfigure}{.45\textwidth}
   \centering
\begin{overpic}[width=0.8\linewidth]{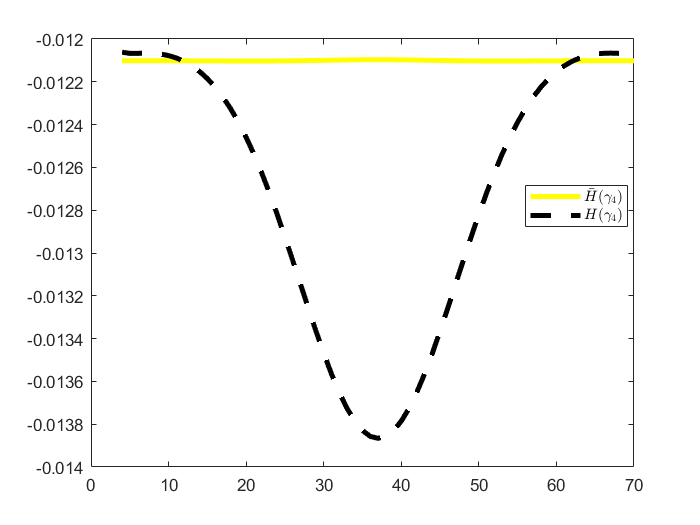}
 \put(45,0){\scriptsize time $n$}
\end{overpic}
  \caption{}
 \end{subfigure}
\caption{Approximation of $\bar H$ along the corresponding trajectories $\gamma_1, \gamma_3, \gamma_4$ from Fig.~\ref{fig:kahan_1}, showing the levels of $\bar H\simeq H+h^2H_2$ (a) which are then compared with $H$ for $\gamma_1$ (b), $\gamma_3$ (c) and $\gamma_4$ (d). }
\label{fig:kahan_2}
\end{figure}

We can show the following connection to the chart $K_1$:
\begin{lemma} \label{specialcont}
The trajectory $\gamma_{h}(n)$, transformed into the chart $K_1$ via 
$$\gamma_{h}^1(n) = \kappa_{21}  (\gamma_{h}(n), h)  $$
for large $\left| n \right|$, lies in                               
$\widehat M_{a,1}$ as well as in $\widehat M_{r,1}$.
\end{lemma}
\begin{proof}
From~\eqref{kappa21} there follows that for sufficiently large $\left|n\right|$, the component $\epsilon_1(n)$ of $\gamma_{h}^1(n)$ is sufficiently small such that $\gamma_{h}^1$, which lies on the invariant manifold $\kappa_{21}(S_{h},h)$, has to be in $N_{a,1}$ for $n <0$, and in $N_{r,1}$ for $n > 0$ respectively, due to the uniqueness of the invariant center manifolds (see Proposition~\ref{centermanifolds}). In particular, observe that, if $h$ is small enough, $\gamma_{h}^1$ reaches an arbitrarily close vicinity of some $p_{a,1}(h_1^*)$ for sufficiently large $n <0$ and of some $p_{r,1}(h_1^*)$ for sufficiently large $n >0$, within $N_{a,1} \subset \widehat M_{a,1}$ and $N_{r,1} \subset \widehat M_{r,1}$ respectively (see also Figure~\ref{fig:blownup} (b)). This finishes the proof.
\end{proof}

The trajectory $\gamma_{h}$ is shown in global blow-up coordinates as $\gamma_{\bar h}$ in Figure~\ref{fig:blownup} (a), in comparison to the ODE trajectory $\bar \gamma_0$ corresponding to $\gamma_{0, 2}$ in $K_2$. 
\begin{figure}[htbp]
        \centering
      \begin{subfigure}[t]{.5\textwidth}
        \centering
  		\begin{overpic}[width=.9\textwidth]{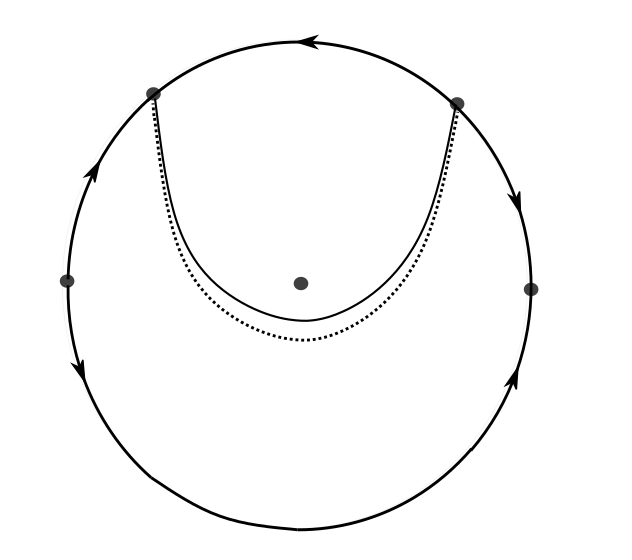}   
                \put(21,78){\scriptsize $\bar p_{\txta}(\bar h)$}
                \put(70,76){\scriptsize $\bar p_{\txtr}(\bar h)$}
                \put(1,46){\scriptsize $\bar q^{\textnormal{in}}(\bar h)$}
                \put(85,45){\scriptsize $\bar q^{\textnormal{out}}(\bar h)$}
                \put(55,44){\small $\bar \gamma_{0}$}
                \put(55,34){\small $\gamma_{\bar h}$}
        \end{overpic}
         \caption{Dynamics on
      $S^{2,+} \times \{0\} \times \{0\} \times \{\bar h \}$, where $S^{2,+}$ denotes the upper hemisphere}
        \label{fig:S2}
		\end{subfigure}%
        \begin{subfigure}[t]{.5\textwidth}
        \centering
  		\begin{overpic}[width=.9\textwidth]{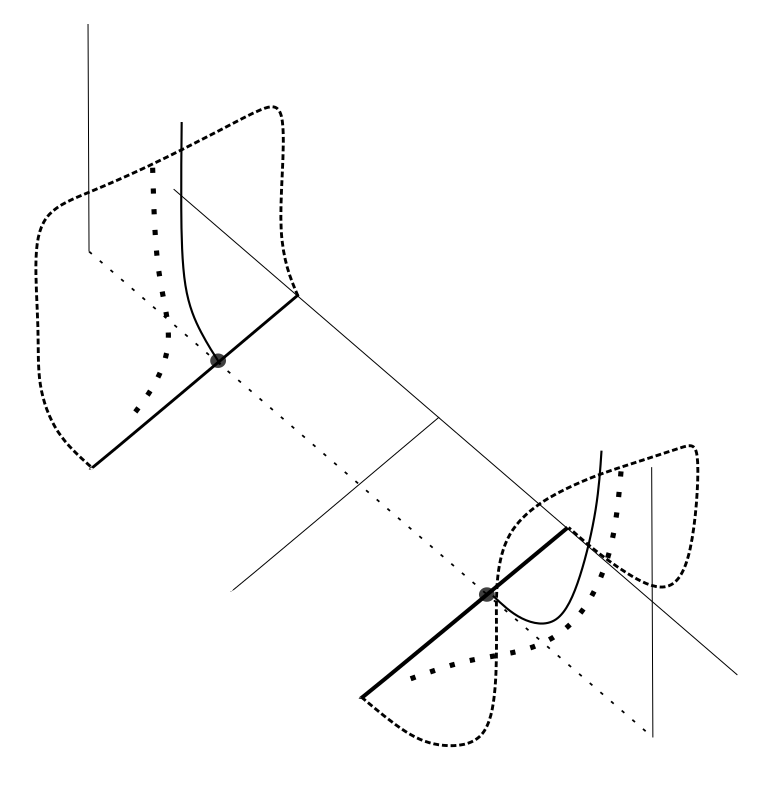}   
           \put(89,12){ \scriptsize $x_1$} 
            \put(4,92){ \scriptsize $\epsilon_1$} 
             \put(27,23){ \scriptsize $h_1$}
             \put(36,78){ \scriptsize $N_{\txtr,1}$} 
             \put(85,47){ \scriptsize $N_{\txta,1}$} 
             \put(38,16){\scriptsize $P_{\txta,1}$}
             \put(20,44){\scriptsize $P_{\txtr,1}$}
              \put(46,26){\tiny $ p_{\txta,1}(h_1^*)$}
                \put(29,54){\tiny $ p_{\txtr,1}(h_1^*)$} 
               \put(72,20){\scriptsize $\gamma_{h}^1$}  
               \put(14,73){\scriptsize $\gamma_{h}^1$}  
               \put(18,87){\scriptsize $\gamma_{0,1}$}  
               \put(72,45){\scriptsize $\gamma_{0,1}$}
        \end{overpic}
        \caption{Dynamics in $K_1$ for $r_1=\lambda_1 =0$}
        \label{fig:NaNr}
		\end{subfigure}
\caption{The trajectory $\gamma_{\bar h}$ in global blow-up coordinates for $r = \bar{\lambda} = 0$ and a fixed $\bar h > 0$ (a), and as $\gamma_{h}^1$ in $K_1$ for $r_1=\lambda_1 = 0$ (b). The figures also show the special ODE solution $\bar \gamma_0$ connecting $\bar p_{r}(\bar h)$ and $\bar p_{a}(\bar h)$ (a), and $ \gamma_{0,1}$ connecting $ p_{r,1}(h_1^*)$ and $ p_{a,1}(h_1^*)$ for fixed $h_1^* > 0$ (b) respectively. In Figure (a), the fixed points $\bar q^{\textnormal{in}}(\bar h)$ and $\bar q^{\textnormal{out}}(\bar h)$, for $\bar \epsilon = 0$, are added, whose existence can be seen in an extra chart (similarly to \cite{ks2011}). In Figure (b), the trajectory $\gamma_{h}^1$ is shown on the attracting center manifold $N_{a,1} \subset \widehat M_{a,1}$ and on the repelling center manifold $N_{r,1} \subset \widehat M_{r,1}$ (see Section~\ref{sec:KahanK1} and Lemma~\ref{specialcont}).}
\label{fig:blownup}
\end{figure}

\subsection{Melnikov computation along the invariant curve} \label{sec:Melnikovnew}

We consider a Melnikov-type computation for the distance between invariant manifolds, which is a discrete time analogue of continuous time results in \cite{ks2001/2} and, for a more general framework, in \cite{Wechselberger2002}. 

Consider an invertible map depending on a parameter $\mu$:
\begin{align} \label{genperturbed}
\begin{array}{r@{\;\,=\;\,}l}
\tilde x & F_1(x,y) +  \mu G_1(x,y,\mu),\\
\tilde y & F_2(x,y) + \mu G_2(x,y,\mu),\\
\tilde \mu & \mu,
\end{array}
\end{align}
where $(x,y) \in \mathbb{R}^2$, and $F=(F_1,F_2)^{\top}$ , $G=(G_1,G_2)^{\top}$ are $C^k$, vector-valued maps, $k\geq 1$. The following theory can be easily extended to  $\mu \in \mathbb{R}^m$, like in~\cite{Wechselberger2002}, but for reasons of clarity we formulate it for $\mu \in \mathbb{R}$.

We formulate the following assumptions:
\begin{enumerate}
\item[(A1)] There exist invariant center manifolds $M_\pm$ of the dynamical system \eqref{genperturbed}, given as graphs of $C^k$-functions $y = g_\pm(x, \mu)$ and intersecting at $\mu = 0$ along the smooth curve 
$$
S=\{(x,y)\in\mathbb{R}^2:y=g(x,0)\},
$$ 
where $g_\pm(x,0)=g(x,0)$.   
\item[(A2)]  Orbits of the map \eqref{genperturbed} with $\mu=0$ passing through a point $(x_0,g(x_0,0))$ on the invariant curve are given by a one-parameter family of solutions $(\gamma_{x_0}(n),0)^\top$ of dynamical system~\eqref{genperturbed} with $\mu=0$, such that $\gamma_{x_0}(n)$ and $G(\gamma_{x_0}(n),0)$ are of a moderate growth when $n\to \pm \infty$ (to be specified later). 
\item[(A3)] There exist solutions $\phi_\pm(n)=(w_{\pm}(n),1)^\top$ of the linearization of~\eqref{genperturbed} along $(\gamma_{x_0}(n),0)^\top$,
\begin{equation} \label{linearizationgamma}
\phi(n+1) = \begin{pmatrix}
\rmD F(\gamma_{x_0}(n)) & G(\gamma_{x_0}(n),0) \\
0 & 1
\end{pmatrix}
\phi(n),
\end{equation}
such that
\begin{equation*}
T_{(\gamma_{x_0}(n),0)^\top} M_\pm = \spn \left\{ \begin{pmatrix}\partial_{x_0} \gamma_{x_0}(n) \\0\end{pmatrix}, \begin{pmatrix} w_\pm(n)\\1\end{pmatrix} \right\},
\end{equation*}
and $w_\pm(n)$ are of a moderate growth (to be specified later) when $n \to \pm \infty$, respectively.
\item[(A4)] The solutions $\psi_{x_0}(n)$ of the \emph{adjoint difference equation}
\begin{equation} \label{adjoint}
\psi(n+1) = \left(\rmD F(\gamma_{x_0}(n))^{\top}\right)^{-1} \psi(n)
\end{equation}
with initial vectors $\psi_{x_0}$ satisfying $\langle\psi_{x_0}(0),\partial_{x_0}\gamma_{x_0}(0)\rangle = 0$, rapidly decay at $\pm \infty$ (the rate of decay to be specified later).
\end{enumerate}
For a given $x_0$, we define $\psi_{x_0}(0)$ to be a unit vector in $\mathbb R^2$ orthogonal to $\partial_{x_0}\gamma_{x_0}(0)$, and set
$$ 
\Sigma = \{(x,y, \mu): (x,y) \in \spn \{\psi_{x_0}(0)\},\; \mu\in \mathbb R \};
$$
the intersections $M_{\pm} \cap \Sigma$ are then given by $(\Delta_{\pm}(\mu) \psi_{x_0}(0), \mu)$, where $\Delta_{\pm}$ are $C^k$-functions.
\medskip

The following proposition is a discrete time analogue of \cite[Proposition 3.1]{ks2001/2}. 
\begin{proposition} \label{formulas}
The first order separation between $M_+$ and $M_-$ at the section $\Sigma$ is given by
\begin{equation} \label{dmu2}
d_{\mu} = -\sum_{n=-\infty}^{\infty} \langle \psi_{x_0}(n+1), G(\gamma_{x_0}(n),0) \rangle.
\end{equation}
\end{proposition}

\begin{proof}
Equations~\eqref{linearizationgamma} and~\eqref{adjoint} read:
\begin{align*}
\psi_{x_0}(n+1) &= \left(\rmD F(\gamma_{x_0}(n))^{\top}\right)^{-1} \psi_{x_0}(n),\\
w_+(n+1) &= \rmD F(\gamma_{x_0}(n)) w_+(n) + G(\gamma_{x_0}(n) ,0), \\
w_-(n+1) &= \rmD F(\gamma_{x_0}(n)) w_-(n) + G(\gamma_{x_0}(n) ,0).
\end{align*}
There follows:
\begin{align*}
& \langle \psi_{x_0}(n+1), w_\pm(n+1) \rangle - \langle \psi_{x_0}(n), w_\pm(n) \rangle \\ 
&\qquad =  \left\langle \left(\rmD F(\gamma_{x_0}(n))^{-1}\right)^{\top} \psi_{x_0}(n), \rmD F(\gamma_{x_0}(n)) w_\pm(n) + G(\gamma_{x_0}(n) ,0) \right\rangle 
- \langle \psi_{x_0}(n), w_\pm(n) \rangle  \\
&\qquad= \langle \psi_{x_0}(n+1), G(\gamma_{x_0}(n),0) \rangle\,.
\end{align*}
Choose initial data $w_\pm(0)=\frac{\rmd \Delta_\pm}{\rmd \mu}(0)\psi_{x_0}(0)$. Assuming that the growth of $w_\pm(n)$ and the decay of $\psi_{x_0}(n)$ at $n\to\pm\infty$, mentioned in (A3) and (A4), are such that 
$$
\lim_{n\to -\infty} \langle \psi_{x_0}(n), w_-(n) \rangle=0, \quad \lim_{n\to +\infty} \langle \psi_{x_0}(n), w_+(n) \rangle=0,
$$
we derive:
$$ 
\frac{\rmd \Delta_-}{\rmd \mu} (0) = \langle \psi_{x_0}(0), w_-(0)\rangle
=  \sum_{n=-\infty}^{-1} \langle \psi_{x_0}(n+1), G(\gamma_{x_0}(n),0) \rangle , 
$$
and
$$ 
\frac{\rmd \Delta_+}{\rmd \mu}  (0) =  \langle \psi_{x_0}(0), w_+(0)\rangle
=  - \sum_{n=0}^{\infty} \langle \psi_{x_0}(n+1), G(\gamma_{x_0}(n),0) \rangle.
$$
From this formula~\eqref{dmu2} follows immediately. 
\end{proof}

We now apply Proposition~\ref{formulas} (or, better to say, its generalization for the case of two parameters $\mu=(r_2,\lambda_2)$) to the Kahan map~\eqref{Kahan_expan} in the rescaling chart $K_2$. First of all, we have to justify Assumptions (A1)--(A4) for this case. Assumption (A1) follows from the fact that for $\mu=(r,\lambda)=0$, the center manifolds $\widehat M_{a,2}$ and $\widehat M_{r,2}$ intersect along the curve $S_h$ given in \eqref{Kahan_invariant}. Assumption (A2) follows from the explicit formula \eqref{Kahan_sol_x0} for the solution $\gamma_{h,x_0}$, as well as 
from formulas \eqref{J} for the functions $\hat J$ and similar formulas for the functions $\hat G$. Assumption (A3) follows from the existence of the center manifolds away from $\mu=(r,\lambda)=0$, established in Proposition \ref{centermanifolds}. Turning to the assumption (A4), we have the following results.
\begin{proposition}\label{prop decaying solution}
For problem~\eqref{Kahan_expan}, the adjoint linear system \eqref{adjoint},
\begin{equation} \label{adjointmap_P0}
\psi(n+1) = \left(\rmD F(\gamma_{h,x_0}(n), h)^{\top}\right)^{-1} \psi(n),
\end{equation}
has the decaying solution
\begin{equation}\label{sol_decay_kahan}
\psi_{h, x_0}(n) = \frac{1}{X(n)} \begin{pmatrix}
-2x_0 - hn \\ 1
\end{pmatrix}, \ n \in \mathbb{Z},
\end{equation} 
where 
\begin{equation}\label{eq X}
X(n) = \prod_{k=0}^{n-1} a(k), \quad
X(-n) = \prod_{k=1}^n (a(-k))^{-1} \;\;{\rm for}\;\; n>0,
\end{equation}
and
\begin{equation}\label{eq fraction a}
a(k) = \frac{1 + h\left(x_0 +  \frac{h}{2}(k+1)\right) + \frac{h^2}{4}}{1 - h\left(x_0 +  \frac{h}{2}k\right) + \frac{h^2}{4}}.
\end{equation}
We have:
\begin{equation} \label{detXn}
\left|X(n)\right| \approx |n|^{4/h^2 + 2}, \ \text{ as } n \to \pm \infty.
\end{equation} 
Here the symbol $\approx$ relates quantities whose quotient has a limit as $n\to\pm\infty$.
\end{proposition}
\begin{proof}

Fix $x_0 \in \mathbb{R}$ and set 
$$ 
A(n)= \rmD F(\gamma_{h, x_0}(n), h).
$$
Let
$$ \Phi(n) = \begin{pmatrix}
\phi_{1,1}(n) & \phi_{1,2}(n) \\
\phi_{2,1}(n) & \phi_{2,2}(n)
\end{pmatrix}
$$
be a fundamental matrix solution of the linear difference equation
$$ 
\phi(n+1) = A(n) \phi(n)
$$
with $\det \Phi(0)=1$. The first column of the fundamental matrix solution $\Phi(n)$ can be found as $\partial_{x_0} \gamma_{h,x_0}$. Using formula~\eqref{Kahan_sol_x0} for $\gamma_{h, x_0}$, we have:
$$ \begin{pmatrix}
\phi_{1,1}(n)  \\
\phi_{2,1}(n) 
\end{pmatrix} = \begin{pmatrix}
1  \\
2x_0 + hn 
\end{pmatrix}.$$

A fundamental solution of the adjoint difference equation 
$$ 
\psi(n+1) =( A^\top(n))^{-1} \psi(n)
$$
is given by
$$ 
\Psi(n) = (\Phi^{\top}(n))^{-1} = \frac{1}{\det \Phi(n)}  \begin{pmatrix}
\phi_{2,2}(n) & -\phi_{2,1}(n) \\
-\phi_{1,2}(n) & \phi_{1,1}(n)
\end{pmatrix}.
$$
Its second column is a solution of the adjoint system as given in \eqref{sol_decay_kahan}, with $X(n)=\det\Phi(n)$. To compute $X(n)$, we observe that from
\begin{eqnarray*}
\Phi(n) & = & A(n-1)A(n-2)\ldots A(0)\Phi(0) \quad {\rm for} \quad n>0,\\
\Phi(0) & = & A(-1)A(-2)\ldots A(-n)\Phi(-n) \quad {\rm for} \quad n>0,
\end{eqnarray*}
and from $\det\Phi(0)=1$, there follows a discrete analogue of Liouville's formula: for $n>0$,
\begin{equation*}
\det \Phi(n) = \prod_{k=0}^{n-1} \det A(k), \quad
\det \Phi(-n) = \prod_{k=1}^{n}(\det A(-k))^{-1},
\end{equation*}
which coincides with \eqref{eq X} with $a(k)=\det A(k)$. Expression \eqref{eq fraction a} for these quantities follows from \eqref{det J prelim}.

To prove the estimate \eqref{detXn}, we observe:
$$ 
a(k) = - \frac{k + \beta}{k - \alpha}\quad{\rm with} \quad 
\alpha = \frac{2}{h^2} \left( 1 - hx_0 + \frac{h^2}{4}\right), \quad \beta = \frac{2}{h^2} \left( 1 + hx_0 + \frac{3 h^2}{4} \right).
$$
Therefore, for $n>0$,
\begin{eqnarray*} 
X(n) & = & (-1)^n\prod_{k=0}^{n-1} \frac{k+\beta}{k-\alpha} = (-1)^n\ \frac{\Gamma(n+\beta)}{\Gamma(n-\alpha)}\ \frac{\Gamma(-\alpha)}{\Gamma(\beta)},\\
X(-n) & =  & (-1)^n\prod_{k=1}^{n} \frac{k+\alpha}{k-\beta} =(-1)^n\ \frac{\Gamma(n+\alpha)}{\Gamma(n-\beta)}\ \frac{\Gamma(-\beta)}{\Gamma(\alpha)}.
\end{eqnarray*}
Using the formula $\Gamma(n+c)\sim n^c \Gamma(n)$ by $n\to+\infty$ (in the sense that the quotient of the both expressions tends to 1), 
we obtain for $n\to+\infty$:
\begin{equation}\label{det Phi} 
|X(n)|, \, |X(-n)| \approx  n^{\alpha+\beta}= n^{4/h^2 +2}.
\end{equation}
This completes the proof.
\end{proof}

With the help of estimates of Proposition \ref{prop decaying solution}, we derive from Proposition \ref{formulas} the following statement:

\begin{proposition}\label{prop:convergence}
For the separation of the center manifolds $\widehat M_{a,2}$ and $\widehat M_{r,2}$, and for sufficiently small $h$, we have the first order expansion
\begin{equation} \label{distancefirstorder}
D_{h,x_0}(r, \lambda) = d_{h,x_0,\lambda} \lambda + d_{h,x_0,r}r + \mathcal{O}(2),
\end{equation}
where $\mathcal{O}(2)$ denotes terms of order $\ge 2$ with respect to $\lambda,r$, and
\begin{equation}\label{eq:dlambda}
d_{h,x_0,\lambda}=-\sum_{n=-\infty}^{\infty} \langle \psi_{h,x_0}(n+1), \hat J(\gamma_{h,x_0}(n),h) \rangle, 
\end{equation}
\begin{equation} \label{eq:dr}
d_{h,x_0,r}=-\sum_{n=-\infty}^{\infty} \langle \psi_{h,x_0}(n+1), \hat G(\gamma_{h,x_0}(n),h) \rangle.
\end{equation}
In particular, convergence of the series in equation~\eqref{eq:dlambda} is obtained for any $h>0$ and convergence of the series in equation~\eqref{eq:dr} is obtained for $0 < h < \sqrt{4/3}$.
\end{proposition}
\begin{proof}
The form of the first order separation follows from Proposition \ref{formulas}. Furthermore, recall from equation~\eqref{J} that 
$$ \hat J(\gamma_h(n),h) = \left(  \frac{h^2}{2} \frac{1}{1-\frac{h^2 n}{2}+\frac{h^2}{4}}, - h \frac{1-\frac{h^2 n}{2}}{1-\frac{h^2 n}{2} + \frac{h^2}{4}}\right) \xrightarrow{n \to \pm \infty} (0, -h).$$
Using Proposition \ref{prop decaying solution}, this yields~\eqref{eq:dlambda} for any $h > 0$.
Note from equation~\eqref{ODE_for_Kahan} that the highest order $n^{\kappa}$ we can obtain in the terms $ \hat G(\gamma_{h}(n),h)$ is $\kappa = 3$ (coming from the term with factor $a_2$) such that for large $\left|n\right|$ we have
$$\langle \psi_{h}(n+1), \hat G(\gamma_{h}(n),h) \rangle = \mathcal{O}\left(n^{-4/h^2-2}n n^3 \right) = \mathcal{O}\left(n^{-4/h^2+2}\right).$$
This means that the convergence in~\eqref{eq:dr} is given for $-4/h^2+2 < -1$ such that the claim follows.
\end{proof}

We are now prepared to show our main result.
\begin{theorem} \label{canard_discrete}
Consider the Kahan discretization for system~\eqref{ODE_for_Kahan}. Then there exist $\epsilon_0, h_0 > 0$ and a smooth function $\lambda_c^h(\sqrt{\epsilon})$ defined on $[0, \epsilon_0]$ such that for $\epsilon \in [0, \epsilon_0]$ and $h \in (0, h_0]$ the following holds:
\begin{enumerate}
\item The attracting slow manifold $S_{a, \epsilon,h}$ and the repelling slow manifold $S_{r, \epsilon,h}$ intersect, i.e. exhibit a maximal canard, if and only if $\lambda = \lambda_c^h(\sqrt{\epsilon})$.
\item 
The function $\lambda_c^h$ has the expansion
$$ \lambda_c^h(\sqrt{\epsilon})= - C \epsilon + \mathcal{O}( \epsilon^{3/2}h),$$ 
where $C$ is given as in~\eqref{constant} (for $a_3 =0$).
\end{enumerate}
\end{theorem}
\begin{proof}
First, we will work in chart $K_2$ and show that the quantities $d_{h,x_0,\lambda}$, $d_{h,x_0,r}$ in~\eqref{eq:dlambda}, \eqref{eq:dr} with $x_0=0$ approximate the quantities  $d_{\lambda}$, $d_{r}$ in \eqref{dr2}, \eqref{dlambda2}(up to change of sign). We prove:
\begin{align}
\sum_{n=-\infty}^{\infty} \langle \psi_{h,0}(n+1), \hat G(\gamma_{h,0}(n),h) \rangle &= \int_{-\infty}^{\infty} \big\langle \psi(t_2), G(\gamma_{0,2}(t_2)) \big\rangle \, \rmd t_2+\mathcal O(h), \label{sum_integral_G} \\
 \sum_{n=-\infty}^{\infty} \langle \psi_{h,0}(n+1), \hat J(\gamma_{h,0}(n),h) \rangle&= \int_{-\infty}^{\infty} \big\langle \psi(t_2),\begin{pmatrix}  0 \\ -1 \end{pmatrix} \big\rangle \, \rmd t_2+\mathcal O(h), \label{sum_integral_J}
\end{align} 
where, recall,
\begin{equation}\label{sols_decay}
\psi_{h, 0}(n) = \frac{1}{X(n)} \begin{pmatrix}  - hn \\ 1 \end{pmatrix}, \quad
 \psi(t_2)=\frac{1}{{\rm e}^{t_2^2/2}}  \begin{pmatrix} -t_2 \\ 1 \end{pmatrix},
\end{equation} 
\begin{equation}\label{sols_x0}
\renewcommand{\arraystretch}{2.2}
\gamma_{h, 0}(n) = \begin{pmatrix} \dfrac{hn}{2} \\ \dfrac{(hn)^2}{4} - \dfrac{1}{2} - \dfrac{h^2}{8} \end{pmatrix}, \quad 
\gamma_{0,2}(t_2) = \begin{pmatrix} \dfrac{t_2}{2} \\ \dfrac{t_2^2}{4} - \dfrac{1}{2}\end{pmatrix},
\end{equation}
the function $\hat J$ is defined as in \eqref{J}, and similar formulas hold true also for the function $\hat G$. Further recall that the Melnikov integrals can be solved explicitly, yielding
\begin{align*}
\int_{-\infty}^{\infty} \langle \psi(t_2), J(\gamma_{0,2}(t_2) \rangle \, \rmd t_2 &= -\int_{-\infty}^{\infty} e^{-t_2^2/2} \, \rmd t_2 =  - \sqrt{2\pi}\,,\\
\int_{-\infty}^{\infty} \langle \psi(t_2), G(\gamma_{0,2}(t) \rangle \,  \rmd t &= \frac{1}{8}\int_{-\infty}^{\infty} (-4 a_5 -(4 a_1 + 2a_2 -2a_4-2a_5 )t_2^2 +a_2 t_2^4) e^{-t_2^2/2} \, \rmd t_2 \\
&=  -C\sqrt{2\pi}\,,
\end{align*}
where $a_i$ and $C$ are as introduced in Section~\ref{sec:contmain} (for $a_3=0$, see~\eqref{ODE_for_Kahan_slow} and \eqref{ODE_for_Kahan}).


We show~\eqref{sum_integral_G} --- the simpler case~\eqref{sum_integral_J} then follows similarly.
We observe:
\begin{enumerate}
\item The remainder of the integral satisfies
$$S(t) := \left(\int_{-\infty}^{-T}+\int_T^{\infty}\right)\langle \psi(t_2),G(\gamma_{0,2}(t_2)\rangle \, \rmd t_2=\mathcal O(T^Me^{-T^2/2}),$$
for $T >0$ and some $M \in \mathbb{N}$.
Hence, we can keep $S(T) = \mathcal O(h^{2-c})$ for any $c>0$ with the choice $T\geq (4 \ln \frac{1}{h})^{1/2}$.

\item For $N=T/h$, we turn to estimate
$$
\hat S(N) := \left(\sum_{n=-\infty}^{-N}+\sum_{n=N}^{\infty}\right) \langle \psi_{h,0}(n+1), \hat G(\gamma_{h,0}(n),h) \rangle .
$$
We denote by $n^*$ the closest integer to $\alpha= 2/h^2 + 1/2$, and recall that $\beta= 2/h^2 + 3/2$. Since 
$$\left|\frac{n^* + \beta}{n^*-\alpha} \right| \geq n^* + \beta \geq 4/h^2,$$
we can write, for all $n \geq 2/h^2 + 3/2$, 
$$ \left| X(n+1) \right| \geq \frac{4}{h^2} \prod_{k=0, k\neq n^*}^n \left| \frac{k+ \beta}{k-\alpha} \right|.$$
Since, with Proposition~\ref{prop decaying solution} the summands of $\hat S(N)$ converge to zero even faster for smaller $h$, we obtain by choosing $N\geq \left \lceil{2/h^2 + 3/2}\right \rceil $, and hence $T \geq 2/h + 5h/2 $, that
$$\left(\sum_{n=-\infty}^{-N}+\sum_{n=N}^{\infty}\right) \langle \psi_{h,0}(n+1), \hat G(\gamma_{h,0}(n),h) \rangle =  \mathcal O(h^2).$$
\item For $T=3/h$, we get by the standard methods the estimate 
$$
\sum_{n=-N}^{N} \langle \psi_{h,0}(n+1), \hat G(\gamma_{h,0}(n),h) \rangle - \int_{-T}^{T} \big\langle \psi(t_2), G(\gamma_{0,2}(t_2)) \big\rangle \, \rmd t_2= \mathcal O(Th^2) = \mathcal O(h). 
$$
\end{enumerate}
Hence, we can conclude that equations~\eqref{sum_integral_G} and~\eqref{sum_integral_J} hold, and, in particular, that $d_{h,0,\lambda}$ and $ d_{h,0,r}$ are bounded away from zero for sufficiently small $h$.
Recall from~\eqref{distancefirstorder} that
\begin{equation*}
D_{h,0}(r, \lambda) = d_{h,0,\lambda} \lambda + d_{h,0,r}r + \mathcal{O}(2), \,,
\end{equation*}
where $D_{h,0}(0,0) = 0$. Hence, the fact that $d_{h,0,\lambda}$ and $ d_{h,r}$ are not zero implies, by the implicit function theorem, that there is a smooth function $ \lambda^{h}(r)$ such that
$$ D_{h,0}(r, \lambda^{h}(r)) = 0$$
in a small neighborhood of $(0,0)$. Transforming back from $K_2$ into original coordinates then proves the first claim.

Furthermore, we obtain
$$  \lambda^{h}(r) = - \frac{d_{h,0,r}}{d_{h,0,\lambda}}r + \mathcal{O}(2) = - C r  + \mathcal{O}(r h)\,.$$
Transformation into original coordinates gives
$$  \lambda_c^{h}(\sqrt{\epsilon}) = - C\epsilon + \mathcal{O}\left(\epsilon^{3/2} h \right)\,.$$
Hence, the second claim follows.
\end{proof}
Numerical computations show that $h_0$ in Theorem~\ref{canard_discrete} does not have to be extremely small but that our results are quite robust for different step sizes.
In Figure~\ref{fig:Melnikov}, we display such computations for the case $a_1 =1$, $a_2=a_4 = a_5 =0$.
In this case, the rescaled Kahan discretization in chart $K_2$ is given by
\begin{align} \label{K2dynamics_Kahan_h3}
\begin{array}{r@{\;\,=\;\,}l}
\tilde x & \dfrac{x - hy + \frac{h}{2} x r - \frac{h^2}{4}  x + \frac{h^2}{2}\lambda}{ 1- hx - \frac{h}{2} r + \frac{h^2}{4}} ,  \\
\tilde y & \dfrac{y - h yx - \frac{h}{2} y r - \frac{h^2}{2}  x^2 - h \lambda + h^2 x \lambda  + h x + \frac{h^2}{2} \lambda r - \frac{h^2}{4}  y}{ 1- hx - \frac{h}{2} r + \frac{h^2}{4}}.
\end{array}
\end{align}
Hence, we obtain
\begin{equation} \label{G}
 \hat G_1(x, y, h) = \frac{h x- \frac{h^2}{2}y - \frac{h^2}{2}x^2}{\left(1 - h x + \frac{h^2}{4}\right)^2}, 
\quad \hat G_2(x, y, h) = \frac{\frac{h^2}{2}x - \frac{h^3}{4}y - \frac{h^3}{4}x^2}{\left(1 - h x + \frac{h^2}{4}\right)^2}.
\end{equation}
For different values of $h$ and $N$ we calculate
$$d_{h,\lambda}(N) :=\sum_{n=-N}^{N-1} \langle \psi_{h}(n+1), \hat J(\gamma_{h}(n),h) \rangle \approx - d_{h, 0, \lambda}\,,$$
and, for the situation of~\eqref{K2dynamics_Kahan_h3} with $\hat G$ as in~\eqref{G},
$$d_{h,r}(N) :=\sum_{n=-N}^{N-1} \langle \psi_{h}(n+1), \hat G(\gamma_{h}(n),h) \rangle \approx -d_{h,0,r}\,,$$
We compare these quantities with the values of the respective continuous-time integrals $d_{\lambda} = - \sqrt{2 \pi}$ and $d_{r} = - \sqrt{2 \pi}/2$ (we have $C=1/2$ in this case).

We observe in Figure~\ref{fig:Melnikov} that the sums converge very fast for relatively small $hN$ in both cases. Additionally, we see that $\left|d_{h,\lambda}(N) - d_{\lambda}  \right|$ is significantly smaller than $\left|d_{h,r}(N) - d_{r}  \right|$ for the same values of $h$. Note that the computations indicate that Theorem~\ref{canard_discrete} holds for the chosen values of $h$ since $d_{h,0,\lambda}  \approx  \sqrt{2 \pi} + \left(d_{\lambda} - d_{h,\lambda}(N) \right)$ is clearly distant from $0$.

\begin{figure}[H]
\centering
\begin{subfigure}{.45\textwidth}
  \centering
  \begin{overpic}[width=1\linewidth]{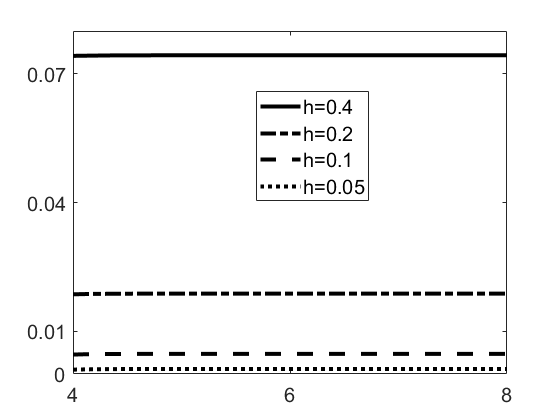}
    \put(54,0){\scriptsize $Nh$}
        \end{overpic}
  \caption{$\left|d_{h, \lambda}(N) - d_{\lambda}  \right|$}
  \label{Melnikov:lambda}
\end{subfigure}%
\begin{subfigure}{.45\textwidth}
  \centering
  \begin{overpic}[width=1\linewidth]{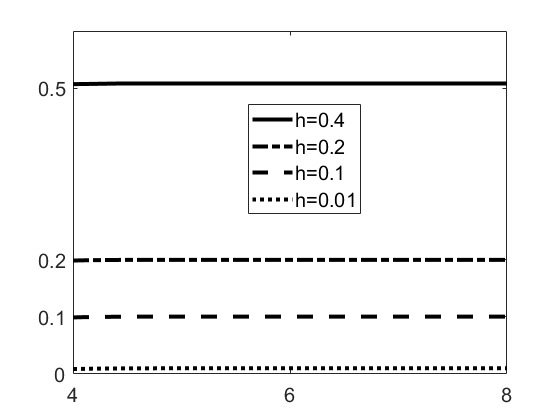}
      \put(54,0){\scriptsize $Nh$}
        \end{overpic}
  \caption{$\left|d_{h,r}(N) - d_{r}  \right|$}
  \label{Melnikov:r}
\end{subfigure}
\caption{ The integral errors (a) $\left|d_{h,\lambda}(N) - d_{\lambda}  \right|$ and (b) $\left|d_{h, r}(N) - d_{r}  \right|$ for different values of $h$ and $N \in \mathbb{N}$.}
\label{fig:Melnikov}
\end{figure}
\subsection{Numerical illustrations for $\epsilon > 0$} \label{sec:numerics}
We illustrate the results by some additional numerics for $\epsilon > 0$, supplemenenting the illustrations of the dynamics in the rescaling chart $K_2$, as given by Figures~\ref{fig:kahan_1} and~\ref{fig:kahan_2}.
Firstly, we consider the simplest case where $a_i =0$ for all $i$, i.e., situation~\eqref{Kahan for slow part} with invariant curve $S_{\epsilon, h}$ \eqref{Kahan invariant curve}. Figure~\ref{fig:kahan_eps} shows different trajectories of the map~\eqref{Kahan for slow part} for $\epsilon =0.1$ and $h=0.02$, illustrating the organization of dynamics around $S_{\epsilon, h}$ analogously to the dynamics of~\eqref{Kahan_model 0} around $S_{h}$ \eqref{Kahan_invariant} (see Figure~\ref{fig:kahan_1}). 

Secondly, we consider the map~\eqref{Kahan for normalform} with $a_1=1$, i.e., a small additional perturbation of the canonical form, similarly to the end of the previous section. We take $\epsilon =0.1$, $h=0.02$ and $\lambda = - (a_1/2) \epsilon$, as a leading order approximation of $\lambda_c^h(\sqrt{\epsilon})$ (see Theorem~\ref{canard_discrete}). In Figure~\ref{fig:Kahan_perturbed}, we observe that the numerics given by the Kahan discretization approximate very well the maximal canard, which slightly deviates from $S_{\epsilon,h}$, again illustrating the organization of dynamics into bounded and unbounded trajectories sperated by the maximal canard. Note that we have chosen $\epsilon =0.1$ to demonstrate the extension up to a relatively large $\epsilon$.

In addition, we consider a model with cubic nonlinearity in order to demonstrate the application of the Kahan method beyond the purely quadratic case. Consider the equation
\begin{align} \label{eq:van_der_Pol}
\begin{array}{r@{\;\,=\;\,}l}
x' & - y  +  x^2 \left(1 + \displaystyle{\frac{x}{3}}\right), \\
y' & \epsilon(x  - \lambda),
\end{array}
\end{align}
as an example of equation~\eqref{normalform},
i.e., $a_3 =1/3$ and $a_i =0, i=1,2,4,5$. Equation~\eqref{eq:van_der_Pol} is the van der Pol equation with constant forcing after transformation around one of the fold points (see \cite[Example 8.1.6]{ku2015}). The Kahan discretization~\eqref{RungeKutta} of this equation yields
\begin{equation}\label{Kahan for van der Pol}
\renewcommand{\arraystretch}{1.9}
\begin{array}{r@{\;\,=\;\,}l}
\dfrac{1}{h}(\tilde x - x) & - \dfrac{1}{2}(y +\tilde y)+x\tilde  x -  \dfrac{x^3+\tilde x^3}{12}(x+\tilde x) + \dfrac{x^2 \tilde x + \tilde x^2 x}{4}, \\ 
\dfrac{1}{h}(\tilde y - y) & \dfrac{\epsilon}{2}(x +\tilde x) -\epsilon \lambda,
\end{array}
\end{equation} 
such that the cubic nonlinearity does not vanish and we do not directly obtain an explicit form. However, we can use~\eqref{Kahan for van der Pol} as a numerical scheme by always taking the unique real solution $\tilde{x}$, closest to $x$ in absolut value, of the cubic polynomial.
\begin{figure}[H]
\centering
 \begin{overpic}[width=0.5\linewidth]{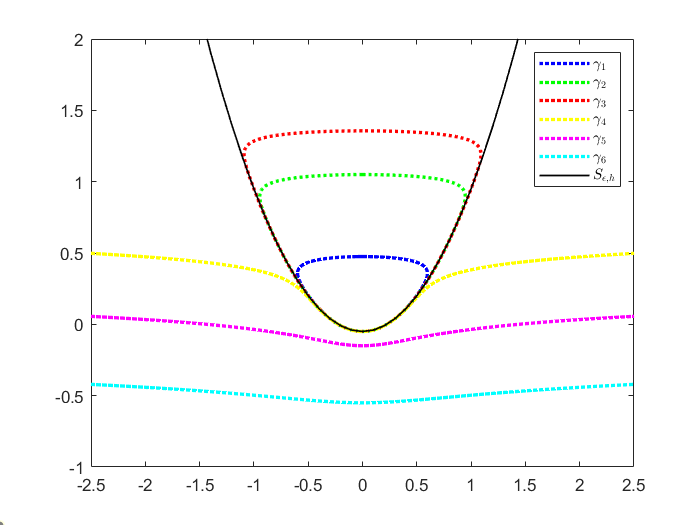}
  \put(54,0){\scriptsize $x$}
    \put(4,40){\scriptsize $y$}
    \end{overpic}
  \caption{Trajectories for the Kahan map~\eqref{Kahan for normalform}, when $a_i =0$ for $i=1,2,4,5$, with $\epsilon =0.1$, $h = 0.02$ and $\lambda =0$, for different initial points: three bounded orbits above the separatrix $S_{\epsilon, h}$, and three unbounded orbits below the separatrix $S_{\epsilon, h}$.}
\label{fig:kahan_eps}
\end{figure}
%
\begin{figure}[H]
\centering
\begin{subfigure}{.45\textwidth}
  \centering
  \begin{overpic}[width=1\linewidth]{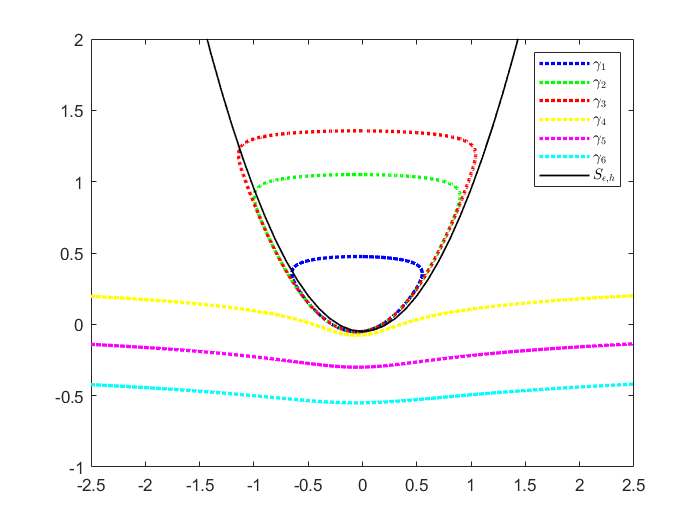}
    \put(54,0){\scriptsize $x$}
    \put(4,40){\scriptsize $y$}
        \end{overpic}
  \caption{}
  \label{fig:Kahan_perturbed_1}
\end{subfigure}%
\begin{subfigure}{.45\textwidth}
  \centering
  \begin{overpic}[width=1\linewidth]{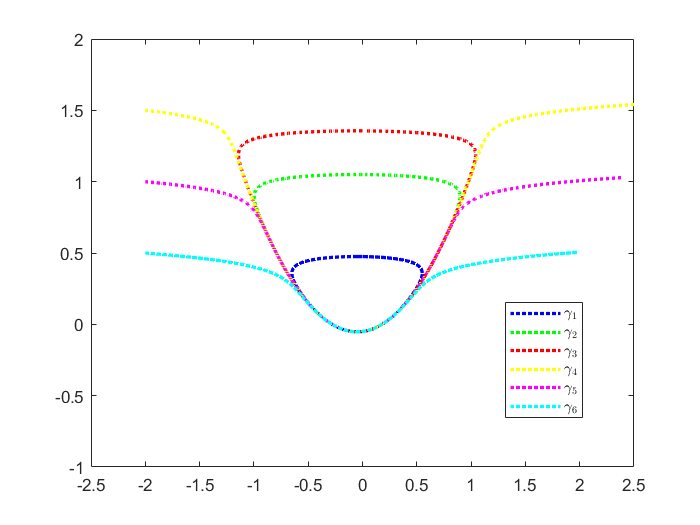}
      \put(54,0){\scriptsize $x$}
      \put(4,40){\scriptsize $y$}
        \end{overpic}
  \caption{}
  \label{fig:Kahan_perturbed_2}
\end{subfigure}
\caption{Trajectories for the Kahan map~\eqref{Kahan for normalform}, when $a_1=1$ and $a_i =0$ for $i=2,4,5$, around maximal canard, taking $\epsilon=0.1$, $h=0.02$ and $\lambda = -(a_1/2)\epsilon$, $a_1=1$: (a) in comparison to symmetric, unperturbed separatrix $S_{\epsilon,h}$, and (b) showing movement along and away from maximal canard.}
\label{fig:Kahan_perturbed}
\end{figure}
In Figure~\ref{fig:kahan_vanderPol}, we illustrate the results of the Kahan discretization~\eqref{Kahan for van der Pol} of the van der Pol equation~\eqref{eq:van_der_Pol}, again for $\epsilon =0.1$ and $h=0.02$, taking $\lambda = - (3 a_3/8) \epsilon$, as a leading order approximation of $\lambda_c(\sqrt{\epsilon})$ (see Theorem~\ref{canard_classic}). Observe that the numerics indicate the existence of a maximal canard, also in this situation, separating bounded, now spiralling, orbits and unbounded orbits. Note that the implementation is based on the fact that the cubic polynomial in $\tilde x$ always has exactly one real solution, which we take as the next value, plus a complex conjugate pair with non-trivial imaginary part. A more general, algebraic analysis extends beyond the scope of this work and is left for additional research.
\begin{figure}[H]
\centering
 \begin{overpic}[width=0.5\linewidth]{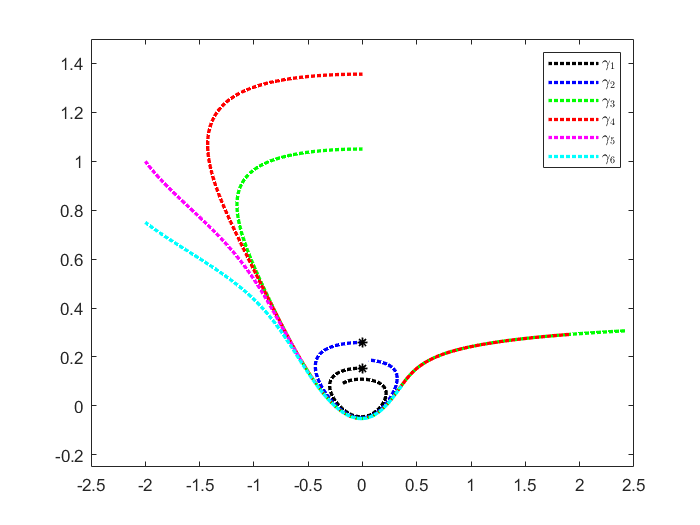}
 \put(54,0){\scriptsize $x$}
      \put(4,40){\scriptsize $y$}
 \end{overpic}
  \caption{Trajectories for the Kahan discretization~\eqref{Kahan for van der Pol} of the transformed van der Pol equation~\eqref{eq:van_der_Pol} with $h = 0.02$ and $\epsilon =0.1$, taking $\lambda = -(3 a_3/8) \epsilon$, $a_3 =1/3$: the orbits $\gamma_1$ and $\gamma_2$ are bounded with initial points $(x_{0},y_{0})$ (black dots) closely above the origin. The other orbits seem to lie beneath a separatrix that would have the role of a maximal canard.}
\label{fig:kahan_vanderPol}
\end{figure}

\begin{figure}[H]
\centering
\begin{subfigure}{.45\textwidth}
  \centering
  \begin{overpic}[width=1\linewidth]{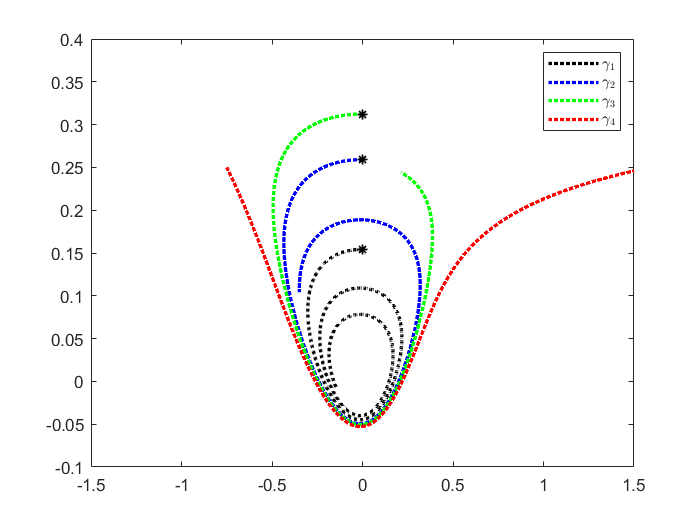}
    \put(54,0){\scriptsize $x$}
    \put(4,40){\scriptsize $y$}
        \end{overpic}
  \caption{$\lambda = -(3 a_3/8) \epsilon$}
\end{subfigure}%
\begin{subfigure}{.45\textwidth}
  \centering
  \begin{overpic}[width=1\linewidth]{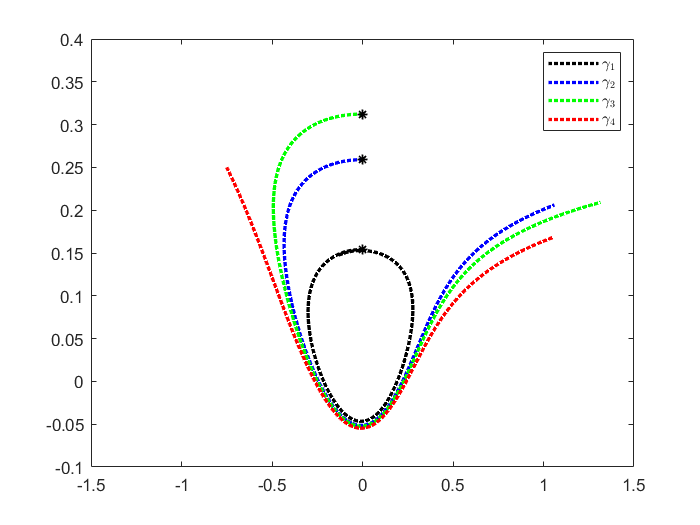}
      \put(54,0){\scriptsize $x$}
      \put(4,40){\scriptsize $y$}
        \end{overpic}
  \caption{$\lambda = -(3 a_3/8) \epsilon + 0.15 \epsilon^{3/2}$}
\end{subfigure}
\caption{Trajectories for the Kahan discretization~\eqref{Kahan for van der Pol} of the transformed van der Pol equation~\eqref{eq:van_der_Pol} with $h = 0.02$ and $\epsilon =0.1$, taking (a) $\lambda = -(3 a_3/8) \epsilon$, $a_3 =1/3$, such that spiralling towards an attractive equilibrium is indicated, and (b) $\lambda = -(3 a_3/8) \epsilon + 0.15 \epsilon^{3/2}$, $a_3 =1/3$, such that a periodic orbit occurs.}
\label{fig:vdP_Hopf}
\end{figure}
Note that the results on maximal canards for perturbations of the canonical form are local and do not make statements on the global stability. The preservation of canards for the van der Pol equation as depicted in Figure~\ref{fig:kahan_vanderPol}  is apparently also of predominantly local nature. Hence, we take a closer look in Figure~\ref{fig:vdP_Hopf}, zooming into a neighbourhood of the inward-spiralling orbits from Figure~\ref{fig:kahan_vanderPol}. 
Here, we observe that the Kahan discretization even seems to capture
the occurrence of a Hopf bifurcation in a neighbourhood of the maximal
canard, as we slightly vary the parameter $\lambda$. Furthermore, the scheme seems to
avoid crossing trajectories near the fold, which do occur as spurious
solutions for some forward numerical methods near maximal canards.
Indeed, there are also robust methods from boundary value problems (BVPs) \cite{DesrochesKrauskopfOsinga2,  Guckenheimeretal2000} and control theory \cite{DurhamMoehlis, JardonKuehn} to track canards for
the van der Pol equation. However, these approaches do not take direct advantage of
the polynomial structure, nor of the particular locally approximately
integrable or symmetry structures of the van der Pol equation. Hence, building on the presented insights for the Kahan method, we consider an analytical treatment of the discretized cubic canard problem an intriguing direction for future work.

\section{Conclusion} \label{sec:conclusion}

Our results show the importance of combining geometric invariants or integrable structures hidden in blow-up coordinates with suitable discretization schemes. Although we have just treated a very low-dimensional fast-slow fold case, one anticipates similar results also to be relevant for various other higher-dimensional singularities and bifurcation points, where blow-up is a standard tool. For example, it is well-known that in the Bogdanov-Takens unfolding one obtains small homoclinic orbits via a hidden integrable structure visible only after re-scaling. A thorough discretization analysis of higher-dimensional canards, similarly to the one at hand, would also deserve further investigation.

From a numerical perspective, forward integration schemes often provide an exploratory perspective to actually detect interesting dynamics or find a suitable invariant solution for fixed parameter values. In several cases these particular forward solutions are then used as starting conditions in numerical continuation techniques~\cite{Dhoogeetal,Doedel_AUTO2007} to study parametric dependence in a setting of BVPs. BVPs have also been successfully adapted to parametrically continue canard-type solutions~\cite{DesrochesKrauskopfOsinga3,DesrochesKrauskopfOsinga2,GuckenheimerKuehn2,KuehnCanLya}. In particular, BVPs for canards turn out to be well-posed with a small numerical error, yet to set up the problem purely by continuation one already needs a very good understanding of phase space for the initial canard orbits. Therefore, a direct numerical integration scheme can be very helpful to automatically yield suitable starting solutions close to a maximal canard.

The Kahan method has mainly turned out to be favorable, since explicit, for quadratic vector fields; hence, in our analysis we have focused on this situation. However, we have seen in the numerical investigations in Section~\ref{sec:numerics} that, by using its implicit form, also non-quadratic problems can be tackled, at least numerically. A further investigation into the dynamical and algebraic properties of the scheme, in particular for cubic nonlinearities, is a highly intriguing research question for the future, in general, and also in particular with respect to geometric multiscale problems as the one presented in this work.
\begin{appendices}
\section{Existence of maximal canards for ODEs} \label{Appendix}

In order to use specific geometric methods in singular perturbation theory, we consider $\epsilon$ and $\lambda$ as variables, writing equation~\eqref{fastequlambda} as
\begin{align} \label{fourvariables}
\begin{array}{r@{\;\,=\;\,}l}
x' & f(x,y, \lambda, \epsilon), \\
y' & \epsilon g(x,y, \lambda, \epsilon), \\
\epsilon' & 0, \\
\lambda' & 0.
\end{array}
\end{align}
Note that in equation~\eqref{normalform} the Jacobi matrix of the vector field in $(x,y, \lambda, \epsilon)$ has a quadruple zero eigenvalue at the origin. A well established way to gain (partial) hyperbolicity at such a singularity is the blow-up technique which replaces the singularity by a manifold on which the dynamics can be desingularized. 
An important technical assumption for this technique is \textit{quasi-homogeneity} of the vector field $f: \mathbb{R}^n \to \mathbb{R}^n$ of the ODE  (cf. \cite[Definition 7.3.2]{ku2015}), which means that there are $(a_1, \dots, a_n) \in \mathbb{N}^n$ and $k \in \mathbb{N}$ such that for every $r \in \mathbb{R}$ and each component $f_j:\mathbb{R}^n \to \mathbb{R}$ of $f$ we have
$$
f_j(r^{a_1} z_1, \dots, r^{a_n} z_n) = r^{k+a_j} f_j(z_1, \dots, z_n). 
$$
The proof of Theorem~\ref{canard_classic} in \cite{ks2011} uses the quasi-homogeneous \textit{blow-up transformation} $\Phi: B \to \mathbb R^4$,
$$ 
x = r \bar x, \quad y =  r^2 \bar y, \quad \epsilon =  r^2 \bar \epsilon, \quad \lambda =  r \bar \lambda, 
$$
where $(\bar x, \bar y,  \bar{\epsilon},\bar \lambda,  r) \in B = S^2 \times [- \kappa, \kappa] \times [0, \rho] $, where
$S^2 = \{(\bar x, \bar y, \bar \epsilon) \, : \, \bar x^2 + \bar y^2 + \bar \epsilon^2 = 1 \}$, with some $\kappa, \rho > 0$. 
We assume that $\rho$ and $\kappa$ sufficiently small, so that the dynamics on $\Phi(B)$ can be described by the normal form approximation.
Let $\overline{X}=\Phi^*(X)$ be the pull-back of the vector field $X$ to $B$. 
The dynamics of $\overline{X}$ on $B$ are analyzed in two charts $K_1$, $K_2$:
\begin{itemize}
\item the \emph{entering and exiting chart} $K_1$ projecting the neighborhood of $(0,1,0)$ on $S^2$ to the plane $\bar y = 1$:
\begin{equation}\label{K1} 
K_1 : \quad x = r_1 x_1, \quad y = r_1^2, \quad \epsilon = r_1^2 \epsilon_1, \quad \lambda = r_1 \lambda_1, 
\end{equation}
\item and the \emph{scaling chart} $K_2$ projecting the neighborhood of $(0,0,1)$ on $S^2$ to the plane $\bar \epsilon = 1$:
\begin{equation} \label{K2}
K_2 : \quad x = r_2 x_2, \quad y = r_2^2 y_2, \quad \epsilon = r_2^2, \quad \lambda = r_2 \lambda_2.
\end{equation}
\end{itemize}
The dynamics in the chart $K_2$ is of a primary interest. Here, the transformed equations admit a time rescaling allowing to divide out a factor $r_2$, which is possible due to the quasi-homogeneity of the leading part of the vector field $X$ (the new time being denoted by $t_2=r_2t$). Upon this operation, equations of motion take the form
\begin{align} \label{K2_hard}
\begin{array}{r@{\;\,=\;\,}l}
x_2' & - y_2 +  x_2^2 + r_2 G_1(x_2, y_2) + \mathcal{O}(r_2(\lambda_2 + r_2)), \\
y_2' & x_2 - \lambda_2 + r_2G_2(x_2, y_2) + \mathcal{O}(r_2(\lambda_2 + r_2)), \\
r_2' & 0, \\
\lambda_2' & 0,
\end{array}
\end{align}
where $G =(G_1, G_2)$ can be written explicitly as
\begin{equation} \label{Gperturbation}
G(x_2, y_2) = \begin{pmatrix}
G_1(x_2, y_2) \\ G_2(x_2,y_2)
\end{pmatrix} = \begin{pmatrix}
a_1 x_2 - a_2 x_2 y_2 + a_3 x_2^3 \\ a_4 x_2^2 + a_5 y_2
\end{pmatrix} .
\end{equation}
On the invariant set $\{r_2 = 0, \lambda_2=0\}$, we have
\begin{equation} \label{K2_easy}
\begin{pmatrix} x_2' \\ y_2' \end{pmatrix} = f(x_2,y_2)= \begin{pmatrix}  - y_2 +  x_2^2  \\ x_2 \end{pmatrix}.
\end{equation}
Let us list some crucially important qualitative features of system \eqref{K2_easy}.
\begin{itemize}
\item As pointed out in \cite[Lemma 3.3]{ks2011}, system~\eqref{K2_easy} possesses an integral of motion
\begin{equation} \label{firstintegral}
H(x_2, y_2) =  \textnormal{e}^{-2 y_2} \left( y_2 - x_2^2 + \frac{1}{2} \right).
\end{equation}
\item Moreover, one can put \eqref{K2_easy} as a generalized Hamiltonian system
\begin{equation} \label{K2 ham}
\begin{pmatrix} x_2'  \\ y_2' \end{pmatrix}=\frac{1}{2}\textnormal{e}^{2y_2}\begin{pmatrix} 0 & 1 \\ -1 & 0\end{pmatrix} \grad H(x_2,y_2).
\end{equation}
\item As a generalized Hamiltonian system, \eqref{K2_easy} preserves the measure $\textnormal{e}^{-2y_2}\rmd x_2\wedge \rmd y_2$. Since the density of an invariant measure is defined up to a multiplication by an integral of motion, the following is an alternative invariant measure:
\begin{equation} \label{K2 invariant_measure}
\mu = \frac{\rmd x \wedge \rmd y}{|y_2 - x_2^2 + \frac{1}{2}|} .
\end{equation}
\item System~\eqref{K2_easy} has an equilibrium of center type at $(0,0)$, surrounded by a family of periodic orbits coinciding with the level curves $\{H(x_2, y_2)=c\}$ for $0<c< \frac{1}{2}$. The level curves for $c<0$ correspond to unbounded solutions. These two regions of the phase plane are separated by the invariant curve $\{H(x_2, y_2)=0\}$, or
\begin{equation}\label{separatrix}
y_2=x_2^2-\frac{1}{2}.
\end{equation}
Thus, we have two alternative characterizations of the separatrix \eqref{separatrix}: on one hand, it is the level set $\{H(x_2, y_2)=0\}$, and on the other hand, it is the singular curve of the invariant measure \eqref{K2 invariant_measure}. 
\item Separatrix \eqref{separatrix} supports a special solution of \eqref{K2_easy}:
\begin{equation} \label{specialsolution}
\renewcommand{\arraystretch}{1.9}
\gamma_{0,2}(t_2) =\begin{pmatrix} x_{0,2}(t_2)\\ y_{0,2}(t_2)\end{pmatrix} = \begin{pmatrix} \dfrac{1}{2} t_2\\ \dfrac{1}{4} t_2^2 - \dfrac{1}{2} \end{pmatrix}, \quad t_2 \in \mathbb{R}.
\end{equation} 
\end{itemize}

Pulled back to the manifold $B$, the special solution $\bar \gamma_0$ connects the endpoint $p_a$ of the critical attracting manifold $S_a$ across the sphere $S^2$ to the endpoint $p_r$ of the critical repelling manifold $S_r$ (see e.g.~\cite[Figure 8.2]{ku2015}).
In other words, the center manifolds $\overline{M}_a$ and $\overline{M}_r$, corresponding to $p_a$ and $p_r$ respectively, and written in chart $K_2$ as $M_{a,2}$ and $M_{r,2}$, intersect along $\gamma_{0,2}$ for $r_2=\lambda_2=0$. 

The difference between $M_{a,2}$ and $M_{r,2}$ for $(r_2,\lambda_2)\neq (0,0)$ is measured by the difference $y_{a,2}(0) - y_{r,2}(0)$, where $\gamma_{a,2}(t) = (x_{a,2}(t), y_{a,2}(t))$ and $\gamma_{r,2}(t) = (x_{r,2}(t), y_{r,2}(t))$ are the trajectories in $M_{a,2}$ and $M_{r,2}$ respectively, for given $r_2,\lambda_2$ with the initial data $x_{a,2}(0) = x_{r,2}(0) = 0$.
This distance can be expressed as \cite[Proposition 3.5]{ks2011}
\begin{equation} \label{DCexpansion}
D(r_2, \lambda_2) = H(0,y_{a,2}(0)) - H(0,y_{r,2}(0)) = d_{r} r_2 + d_{\lambda} \lambda_2 + \mathcal{O}(2)\,,
\end{equation}
where
\begin{align}
 d_{r} &= \int_{-\infty}^{\infty} \big\langle \grad H(\gamma_{0,2}(t_2)), G(\gamma_{0,2}(t_2)) \big\rangle \, \rmd t_2, \label{dr2} \\
 d_{\lambda} &= \int_{-\infty}^{\infty} \big\langle \grad H(\gamma_{0,2}(t_2)),\begin{pmatrix}  0 \\ -1 \end{pmatrix} \big\rangle \, \rmd t_2\label{dlambda2}
\end{align} 
are the respective Melnikov integrals. Since $d_{\lambda} \neq 0$, one concludes by the implicit function theorem that for sufficiently small $r_2$ there exists $\lambda_2$ such that the manifolds $M_{a,2}$ and $M_{r,2}$ intersect. Transforming back into the original variables yields Theorem~\ref{canard_classic}.

It will be important for us that formulas \eqref{dr2}, \eqref{dlambda2} admit also a non-Hamiltonian expression given in \cite{Wechselberger2002}, where $\grad H(\gamma_{0,2}(t_2))$ is replaced by
\begin{equation}\label{K2 psi}
\psi(t_2)=2\textnormal{e}^{-2y_{0,2}(t_2)}\begin{pmatrix} -y_{0,2}'(t_2) \\ x_{0,2}'(t_2) \end{pmatrix}=\textnormal{e}^{-2y_{0,2}(t_2)}\begin{pmatrix} -t_2 \\ 1 \end{pmatrix}.
\end{equation}
The function $\psi(t_2)$ admits a more intrinsic interpretation as the only exponentially decaying solution of the adjoint system for the system \eqref{K2_easy}  linearized along the solution $\gamma_{0,2}(t_2)$,
\begin{equation} \label{contadjoint}
 \psi' = -\rmD f (\gamma_{0,2}(t_2))^{\top} \psi,
\end{equation}
while the expression $2y_{0,2}(t_2)=t_2^2/2$ in the exponent is interpreted as
\begin{equation}
\frac{t_2^2}{2}=\int_0^{t_2} \tr \rmD f(\gamma_{0,2}(\tau))\rmd \tau,
\end{equation}
the matrix of the system \eqref{K2_easy} linearized along the solution $\gamma_{0,2}(t_2)$ being given by 
\begin{equation}
\rmD f(\gamma_{0,2}(t_2)) = \begin{pmatrix} 2x_{0,2}(t_2) & -1 \\ 1 & 0 \end{pmatrix}
                                       = \begin{pmatrix} t_2 & -1 \\ 1 & 0 \end{pmatrix}.
\end{equation}
\end{appendices}

\bibliographystyle{abbrv}
\bibliography{mybibfile}

\end{document}